\documentclass{amsproc}
\usepackage{amsmath}
\usepackage{enumerate}
\usepackage{amsmath,amsthm,amscd,amssymb}

\usepackage{latexsym}
\usepackage{upref}
\usepackage{verbatim}

\usepackage[mathscr]{eucal}
\usepackage{dsfont}

\usepackage{graphicx}
\usepackage[colorlinks,hyperindex,hypertex]{hyperref}

\newtheorem{theorem}{Theorem}

\newtheorem{lemma}[theorem]{Lemma}
\newtheorem{definition}[theorem]{Definition}

\chardef\bslash=`\\ 

\newcommand{\wh}{\widehat}

\newcommand{\dA}{{\dot A}}

\newcommand{\dom}{\text{\rm{Dom}}}

\newcommand{\calH}{{\mathcal H}}

\newcommand{\calR}{{\mathcal R}}

   \def\sN{{\mathfrak N}}

\def\bA{{\mathbb A}}      \def\dC{{\mathbb C}}

      \def\dR{{\mathbb R}}

   \def\dZ{{\mathbb Z}}

   \def\cH{{\mathcal H}}

\def\RE{{\rm Re\,}}
\def\Ker{{\rm Ker\,}}

\def\wh{\hat}

\def\uphar{{\upharpoonright\,}}

\DeclareMathOperator{\IM}{Im}

\DeclareMathOperator{\Ext}{Ext}

\newcommand{\eval}[2][\right]{\relax
  \ifx#1\right\relax \left.\fi#2#1\rvert}

\begin{document}

\title[On realization of the original Weyl-Titchmarsh functions]{On realization of the original  Weyl-Titchmarsh functions by   Shr\"odinger  L-systems}

\author{S. Belyi}
\address{Department of Mathematics\\ Troy University\\
Troy, AL 36082, USA\\
}
\curraddr{}
\email{sbelyi@troy.edu}


\author{E. Tsekanovski\u i}
\address{Department of Mathematics\\ Niagara University, New York
14109\\ USA}
\email{tsekanov@niagara.edu}

\subjclass{Primary 47A10; Secondary 47N50, 81Q10}
\date{DD/MM/2004}

\dedicatory{Dedicated with great pleasure to Henk de Snoo on the occasion of his 75-th birthday}

\keywords{L-system, transfer function, impedance function,  Herglotz-Nevanlinna function, Stieltjes function, Shr\"odinger  operator, Weyl-Titch\-marsh function}

\begin{abstract}
We study realizations generated by the original Weyl-Titch\-marsh functions $m_\infty(z)$ and  $m_\alpha(z)$.   It is shown that the Herglotz-Nevanlinna functions $(-m_\infty(z))$ and $(1/m_\infty(z))$ can be realized as the impedance functions of  the corresponding   Shr\"odinger  L-systems sharing the same main dissipative operator. These L-systems are presented explicitly and  related to Dirichlet and Neumann boundary problems. Similar results but related to the mixed boundary problems are derived for the Herglotz-Nevanlinna functions $(-m_\alpha(z))$ and $(1/m_\alpha(z))$. We also obtain some additional properties of these realizations in the case when the minimal symmetric Shr\"odinger  operator is non-negative. In addition to that we state and prove the uniqueness realization criteria for Shr\"odinger L-systems with equal boun\-dary parameters.  A  condition for two Shr\"odinger L-systems  to share the same main operator is established as well. Examples that illustrate the obtained results are presented in the end of the paper.

 \end{abstract}

\maketitle

\tableofcontents

\section{Introduction}\label{s1}


In the current paper we consider L-systems  with dissipative Shr\"odinger operators. For the sake of brevity we will refer to these L-systems as \textit{Shr\"odinger L-systems} for the rest of the manuscript.   The formal definition, exposition and discussions of general and Shr\"odinger L-systems are presented in Sections \ref{s2} and  \ref{s4}. We capitalize on the fact that all Shr\"odinger L-systems $\Theta_{\mu,h}$ form a two-parametric family whose members are uniquely defined by a real-valued parameter $\mu$ and a complex boundary value $h$ ($\IM h>0$) of the main dissipative operator.

The focus of the paper is set on two classical objects related to a Shr\"odinger  operator:  the \textit{original Weyl-Titchmarsh function} $m_\infty(z)$ and its \textit{linear-fractional transformation} $m_\alpha(z)$ given by \eqref{e-59-LFT} that was introduced and studied in \cite{W},  \cite{W10}, \cite{Levitan}, \cite{Ti62}. It is well known (see \cite{Ti62}, \cite{Levitan}) that $(-m_\infty(z))$ and $(1/m_\infty(z))$ as well as
$(-m_\alpha(z))$ and $(1/m_\alpha(z))$ are  the Herglotz-Nevanlin\-na functions. In Section \ref{s5} we show that the Herglotz-Nevanlinna functions  $(-m_\infty(z))$ and $(1/m_\infty(z))$ can be realized as the impedance function of  Shr\"odinger L-systems  $\Theta_{0,i}$ and $\Theta_{\infty,i}$, respectively (see Theorems \ref{t-6} and \ref{t-7}). Moreover, these two realizing L-system share the same main dissipative Shr\"odinger operator and are connected to Dirichlet \eqref{e-42-quasi} and Neumann \eqref{e-53-quasi} boundary problems, respectively. In Section \ref{s6} we treat the realization of the Herglotz-Nevanlinna functions $(-m_\alpha(z))$ and $(1/m_\alpha(z))$ that are linear-fractional transformations of $m_\infty(z)$ described in details in \cite{Levitan}. As a result we obtain a one-parametric families of realizing Shr\"odinger L-systems $\Theta_{\tan\alpha, i}$ and $\Theta_{(-\cot\alpha), i}$, respectively. In Section \ref{s7} we narrow down the realization results from the previous two sections to the class of Shr\"odinger L-systems that are based on non-negative symmetric Shr\"odinger  operator to obtain additional properties. In particular, in Theorem \ref{t-14} we describe the cases when the realizing Shr\"odinger L-systems are accretive. Moreover, it turns out  that the quasi-kernel $\hat A_0$ of $\RE\bA_{0, i}$ in the constructed realizing L-system $\Theta_{0,i}$ corresponds to the Friedrich's extension  while the quasi-kernel $\hat A_\infty$ of $\RE\bA_{\infty, i}$ in $\Theta_{\infty,i}$ corresponds to the Krein-von Neumann extension of our non-negative symmetric operator $\dA$ only in the case when $m_\infty(-0)=0$.  Section \ref{s8} raises and answers the uniqueness  questions of realization by a Shr\"odinger L-systems. After giving the general definition of two equal L-systems,  we state and prove the criteria for two Shr\"odinger L-systems  to be equal. Precisely,  Theorem \ref{t-17} is saying that two Shr\"odinger L-systems with the same underlying parameters $h$ and $\mu$ are equal if and only if their impedance functions match.
Then  we generalize this result and establish a  condition for two Shr\"odinger L-systems to share the same main operator. In Theorem \ref{t-19} we show that two Shr\"odinger L-systems with the same parameter $h$  share the same main operator $T_h$ if and only if their impedance functions are connected by the Donoghue transform \eqref{e-87-Don}. The paper is concluded with two examples that illustrate  the main results and concepts.

The present work is a further development of the theory of open physical systems conceived by M.~Liv\u sic in \cite{Lv2}.

\section{Preliminaries}\label{s2}

For a pair of Hilbert spaces $\calH_1$, $\calH_2$ we denote by
$[\calH_1,\calH_2]$ the set of all bounded linear operators from
$\calH_1$ to $\calH_2$. Let $\dA$ be a closed, densely defined,
symmetric operator in a Hilbert space $\calH$ with inner product
$(f,g),f,g\in\calH$. Any non-symmetric operator $T$ in $\cH$ such that
\[
\dA\subset T\subset\dA^*
\]
is called a \textit{quasi-self-adjoint extension} of $\dA$.

 Consider the rigged Hilbert space (see \cite{Ber}, \cite{ABT})
$\calH_+\subset\calH\subset\calH_- ,$ where $\calH_+ =\dom(\dA^*)$ and
\begin{equation}\label{108}
(f,g)_+ =(f,g)+(\dA^* f, \dA^*g),\;\;f,g \in \dom(A^*).
\end{equation}
Let $\calR$ be the \textit{\textrm{Riesz-Berezansky   operator}} $\calR$ (see \cite{Ber}, \cite{ABT}) which maps $\mathcal H_-$ onto $\mathcal H_+$ such
 that   $(f,g)=(f,\calR g)_+$ ($\forall f\in\calH_+$, $g\in\calH_-$) and
 $\|\calR g\|_+=\| g\|_-$.
 Note that
identifying the space conjugate to $\calH_\pm$ with $\calH_\mp$, we
get that if $\bA\in[\calH_+,\calH_-]$, then
$\bA^*\in[\calH_+,\calH_-].$
An operator $\bA\in[\calH_+,\calH_-]$ is called a \textit{self-adjoint
bi-extension} of a symmetric operator $\dA$ if $\bA=\bA^*$ and $\bA
\supset \dA$.
Let $\bA$ be a self-adjoint
bi-extension of $\dA$ and let the operator $\hat A$ in $\cH$ be defined as follows:
\[
\dom(\hat A)=\{f\in\cH_+:\bA f\in\cH\}, \quad \hat A=\bA\uphar\dom(\hat A).
\]
The operator $\hat A$ is called a \textit{quasi-kernel} of a self-adjoint bi-extension $\bA$ (see \cite{TSh1}, \cite[Section 2.1]{ABT}).
According to the von Neumann Theorem (see \cite[Theorem 1.3.1]{ABT}) the domain of $\wh A$, a self-adjoint extension of $\dA$,  can be expressed as
\begin{equation}\label{DOMHAT}
\dom(\hat A)=\dom(\dA)\oplus(I+U)\sN_{i},
\end{equation}
where von Neumann's parameter $U$ is a $(\cdot)$ (and $(+)$)-isometric operator from $\sN_i$ into $\sN_{-i}$  and $$\sN_{\pm i}=\Ker (\dA^*\mp i I)$$ are the deficiency subspaces of $\dA$.
 A self-adjoint bi-extension $\bA$ of a symmetric operator $\dA$ is called \textit{t-self-adjoint} (see \cite[Definition 4.3.1]{ABT}) if its quasi-kernel $\hat A$ is self-adjoint operator in $\calH$.
An operator $\bA\in[\calH_+,\calH_-]$  is called a \textit{quasi-self-adjoint bi-extension} of an operator $T$ if $\bA\supset T\supset \dA$ and $\bA^*\supset T^*\supset\dA.$  We will be mostly interested in the following type of quasi-self-adjoint bi-extensions.
Let $T$ be a quasi-self-adjoint extension of $\dA$ with nonempty resolvent set $\rho(T)$. A quasi-self-adjoint bi-extension $\bA$ of an operator $T$ is called (see \cite[Definition 3.3.5]{ABT}) a \textit{($*$)-extension } of $T$ if $\RE\bA$ is a
t-self-adjoint bi-extension of $\dA$.
In what follows we assume that $\dA$ has deficiency indices $(1,1)$. In this case it is known \cite{ABT} that every  quasi-self-adjoint extension $T$ of $\dA$  admits $(*)$-extensions.
The description of all $(*)$-extensions via Riesz-Berezansky   operator $\calR$ can be found in \cite[Section 4.3]{ABT}.

Recall that a linear operator $T$ in a Hilbert space $\calH$ is called \textbf{accretive} \cite{Ka} if $\RE(Tf,f)\ge 0$ for all $f\in \dom(T)$.  We call an accretive operator $T$
\textbf{$\beta$-sectorial} \cite{Ka} if there exists a value of $\beta\in(0,\pi/2)$ such that
\begin{equation}\label{e8-29}
    (\cot\beta) |\IM(Tf,f)|\le\,\RE(Tf,f),\qquad f\in\dom(T).
\end{equation}
We say that the angle of sectoriality $\beta$ is \textbf{exact} for a $\beta$-sectorial
operator $T$ if $$\tan\beta=\sup_{f\in\dom(T)}\frac{|\IM(Tf,f)|}{\RE(Tf,f)}.$$
An accretive operator is called \textbf{extremal accretive} if it is not $\beta$-sectorial for any $\beta\in(0,\pi/2)$.
 A $(*)$-extension $\bA$ of $T$ is called \textbf{accretive} if $\RE(\bA f,f)\ge 0$ for all $f\in\cH_+$. This is
equivalent to that the real part $\RE\bA=(\bA+\bA^*)/2$ is a nonnegative self-adjoint bi-extension of $\dA$.

The following definition is a ``lite" version of the definition of L-system given for a scattering L-system with
 one-dimensional input-output space. It is tailored for the case when the symmetric operator of an L-system has deficiency indices $(1,1)$. The general definition of an L-system can be found in \cite[Definition 6.3.4]{ABT} (see also \cite{BHST1} for a non-canonical version).
\begin{definition} 
 An
array
\begin{equation}\label{e6-3-2}
\Theta= \begin{pmatrix} \bA&K&\ 1\cr \calH_+ \subset \calH \subset
\calH_-& &\dC\cr \end{pmatrix}
\end{equation}
 is called an \textbf{{L-system}}   if:
\begin{enumerate}
\item[(1)] {$T$ is a dissipative ($\IM(Tf,f)\ge0$, $f\in\dom(T)$) quasi-self-adjoint extension of a symmetric operator $\dA$ with deficiency indices $(1,1)$};
\item[(2)] {$\mathbb  A$ is a   ($\ast $)-extension of  $T$};
\item[(3)] $\IM\bA= KK^*$, where $K\in [\dC,\calH_-]$ and $K^*\in [\calH_+,\dC]$.
\end{enumerate}
\end{definition}
  Operators $T$ and $\bA$ are called a \textit{main and state-space operators respectively} of the system $\Theta$, and $K$ is  a \textit{channel operator}.
It is easy to see that the operator $\bA$ of the system  \eqref{e6-3-2}  is such that $\IM\bA=(\cdot,\chi)\chi$, $\chi\in\calH_-$ and pick $K c=c\cdot\chi$, $c\in\dC$ (see \cite{ABT}).
  A system $\Theta$ in \eqref{e6-3-2} is called \textit{minimal} if the operator $\dA$ is a prime operator in $\calH$, i.e., there exists no non-trivial reducing invariant subspace of $\calH$ on which it induces a self-adjoint operator. Minimal L-systems of the form \eqref{e6-3-2} with  one-dimensional input-output space were also considered in \cite{BMkT}.

We  associate with an L-system $\Theta$ the  function
\begin{equation}\label{e6-3-3}
W_\Theta (z)=I-2iK^\ast (\mathbb  A-zI)^{-1}K,\quad z\in \rho (T),
\end{equation}
which is called the \textbf{transfer  function} of the L-system $\Theta$. We also consider the  function
\begin{equation}\label{e6-3-5}
V_\Theta (z) = K^\ast (\RE\bA - zI)^{-1} K,
\end{equation}
that is called the
\textbf{impedance function} of an L-system $ \Theta $ of the form (\ref{e6-3-2}).  The transfer function $W_\Theta (z)$ of the L-system $\Theta $ and function $V_\Theta (z)$ of the form (\ref{e6-3-5}) are connected by the following relations valid for $\IM z\ne0$, $z\in\rho(T)$,
\begin{equation*}\label{e6-3-6}
\begin{aligned}
V_\Theta (z) &= i [W_\Theta (z) + I]^{-1} [W_\Theta (z) - I],\\
W_\Theta(z)&=(I+iV_\Theta(z))^{-1}(I-iV_\Theta(z)).
\end{aligned}
\end{equation*}
An L-system $\Theta $ of the form \eqref{e6-3-2} is called an \textbf{accretive L-system} (\cite{BT10}, \cite{DoTs}) if its state-space operator  operator $\bA$ is accretive, that is $\RE(\bA f,f)\ge 0$ for all $f\in \calH_+$.
 An accretive  L-system  is called  \textbf{sectorial} if the operator  $\bA$ is sectorial, i.e., satisfies \eqref{e8-29} for some $\beta\in(0,\pi/2)$ and all $f\in\cH_+$. 

Now let us consider a minimal L-system $\Theta$ of the form \eqref{e6-3-2}. 
Let also
\begin{equation}\label{e-62-alpha}
\Theta_\alpha= \begin{pmatrix} \bA_\alpha&K_\alpha&\ 1\cr \calH_+ \subset \calH \subset
\calH_-& &\dC\cr \end{pmatrix},\quad \alpha\in[0,\pi),
\end{equation}
 be a one-parametric family of L-systems such that
 \begin{equation}\label{e-63-alpha}
    W_{\Theta_\alpha}(z)=W_\Theta(z)\cdot (-e^{2i\alpha}),\quad \alpha\in[0,\pi).
 \end{equation}
The existence and structure of $\Theta_\alpha$ were described in details in \cite[Section 8.3]{ABT}. In particular, it was shown that $\Theta$ and $\Theta_\alpha$ share the same main operator $T$ and that
\begin{equation}\label{e-64-alpha}
    V_{\Theta_\alpha}(z)=\frac{\cos\alpha+(\sin\alpha) V_\Theta(z)}{\sin\alpha-(\cos\alpha) V_\Theta(z)}.
\end{equation}
Formula \eqref{e-64-alpha} defines \textbf{Donoghue transform} of the function $V_\Theta(z)$ (see \cite[Section 8.3]{ABT} for a more general definition).

A scalar function $V(z)$ is called the Herglotz-Nevanlinna function if it is holomorphic on ${\dC \setminus \dR}$, symmetric with respect to the real axis, i.e., $V(z)^*=V(\bar{z})$, $z\in {\dC \setminus \dR}$, and if it satisfies the positivity condition $\IM V(z)\geq 0$,  $z\in \dC_+$.
 The class of all Herglotz-Nevanlinna functions, that can be realized as impedance functions of L-systems, and connections with Weyl-Titchmarsh functions can be found in \cite{ABT}, \cite{BMkT},  \cite{DMTs},  \cite{GT} and references therein.
The following definition  can be found in \cite{KK74}.
A scalar Herglotz-Nevanlinna function $V(z)$ is a \textit{Stieltjes function} if it is holomorphic in $\Ext[0,+\infty)$ and
\begin{equation}\label{e4-0}
\frac{\IM[zV(z)]}{\IM z}\ge0.
\end{equation}
It is known \cite{KK74} that a Stieltjes function  $V(z)$  admits the following integral representation
\begin{equation}\label{e8-94}
V(z) =\gamma+\int\limits_0^\infty\frac {dG(t)}{t-z},
\end{equation}
where $\gamma\ge0$ and $G(t)$ is a non-decreasing on $[0,+\infty)$  function such that
$\int^\infty_0\frac{dG(t)}{1+t}<\infty.$  We are going to focus on the class $S_0(R)$  (see \cite{BT10}, \cite{DoTs}, \cite{ABT}) of scalar Stieltjes functions, whose definition is the following.
A scalar Stieltjes function $V(z)$ is said to be a member of the \textbf{class $S_0(R)$}\index{Class!$S_0(R)$} if the measure $G(t)$ in  representation \eqref{e8-94} is of unbounded variation.
It was shown in \cite{ABT} (see also \cite{BT10}) that such a function $V(z)$ can be realized as the impedance function of an accretive  L-system $\Theta$ of the form \eqref{e6-3-2} with a densely defined symmetric operator if and only if it belongs to the class $S_0(R)$.

\section{L-systems with Schr\"odinger operator and their impedance functions}\label{s4}

Let $\calH=L_2[\ell,+\infty)$, $\ell\ge0$, and $l(y)=-y^{\prime\prime}+q(x)y$, where $q$ is a real locally summable on  $[\ell,+\infty)$ function. Suppose that the symmetric operator
\begin{equation}
\label{128}
 \left\{ \begin{array}{l}
 \dA y=-y^{\prime\prime}+q(x)y \\
 y(\ell)=y^{\prime}(\ell)=0 \\
 \end{array} \right.
\end{equation}
has deficiency indices (1,1). Let $D^*$ be the set of functions locally absolutely continuous together with their first derivatives such that $l(y) \in L_2[\ell,+\infty)$. Consider $\calH_+=\dom(\dA^*)=D^*$ with the scalar product
$$(y,z)_+=\int_{\ell}^{\infty}\left(y(x)\overline{z(x)}+l(y)\overline{l(z)}
\right)dx,\;\; y,\;z \in D^*.$$ Let $\calH_+ \subset L_2[\ell,+\infty) \subset \calH_-$ be the corresponding triplet of Hilbert spaces. Consider the operators
\begin{equation}\label{131}
 \left\{ \begin{array}{l}
 T_hy=l(y)=-y^{\prime\prime}+q(x)y \\
 hy(\ell)-y^{\prime}(\ell)=0 \\
 \end{array} \right.
           ,\quad  \left\{ \begin{array}{l}
 T^*_hy=l(y)=-y^{\prime\prime}+q(x)y \\
 \overline{h}y(\ell)-y^{\prime}(\ell)=0 \\
 \end{array} \right.,
\end{equation}
where $\IM h>0$.
Let  $\dA$ be a symmetric operator  of the form \eqref{128} with deficiency indices (1,1), generated by the differential operation $l(y)=-y^{\prime\prime}+q(x)y$. Let also $\varphi_k(x,\lambda) (k=1,2)$ be the solutions of the following Cauchy problems:
$$\left\{ \begin{array}{l}
 l(\varphi_1)=\lambda \varphi_1 \\
 \varphi_1(\ell,\lambda)=0 \\
 \varphi'_1(\ell,\lambda)=1 \\
 \end{array} \right., \qquad
\left\{ \begin{array}{l}
 l(\varphi_2)=\lambda \varphi_2 \\
 \varphi_2(\ell,\lambda)=-1 \\
 \varphi'_2(\ell,\lambda)=0 \\
 \end{array} \right.. $$
It is well known \cite{Na68}, \cite{Levitan} that there exists a function $m_\infty(\lambda)$  introduced by H.~Weyl \cite{W}, \cite{W10} for which
$$\varphi(x,\lambda)=\varphi_2(x,\lambda)+m_\infty(\lambda)
\varphi_1(x,\lambda)$$ belongs to $L_2[\ell,+\infty)$. The function $m_\infty(\lambda)$ is not a Herglotz-Nevanlinna function but $(-m_\infty(\lambda))$ and $(1/m_\infty(\lambda))$ are.

Now we shall construct an L-system based on a non-self-adjoint Schr\"odin\-ger operator $T_h$ with $\IM h>0$.  It  was shown in \cite{ArTs0}, \cite{ABT} that  the set of all ($*$)-extensions of a non-self-adjoint Schr\"odinger operator $T_h$ of the form \eqref{131} in $L_2[\ell,+\infty)$ can be represented in the form
\begin{equation}\label{137}
\begin{split}
&\bA_{\mu, h}\, y=-y^{\prime\prime}+q(x)y-\frac {1}{\mu-h}\,[y^{\prime}(\ell)-
hy(\ell)]\,[\mu \delta (x-\ell)+\delta^{\prime}(x-\ell)], \\
&\bA^*_{\mu, h}\, y=-y^{\prime\prime}+q(x)y-\frac {1}{\mu-\overline h}\,
[y^{\prime}(\ell)-\overline hy(\ell)]\,[\mu \delta
(x-\ell)+\delta^{\prime}(x-\ell)].
\end{split}
\end{equation}
Moreover, the formulas \eqref{137} establish a one-to-one correspondence between the set of all ($*$)-extensions of a Schr\"odinger operator $T_h$ of the form \eqref{131} and all real numbers $\mu \in [-\infty,+\infty]$. One can easily check that the ($*$)-extension $\bA_{\mu, h}$ in \eqref{137} of the non-self-adjoint dissipative Schr\"odinger operator $T_h$, ($\IM h>0$) of the form \eqref{131} satisfies the condition
\begin{equation*}\label{145}
\IM\bA_{\mu, h}=\frac{\bA_{\mu, h} - \bA^*_{\mu, h}}{2i}=(.,g_{\mu, h})g_{\mu, h},
\end{equation*}
where
\begin{equation}\label{146}
g_{\mu, h}=\frac{(\IM h)^{\frac{1}{2}}}{|\mu - h|}\,[\mu\delta(x-\ell)+\delta^{\prime}(x-\ell)]
\end{equation}
and $\delta(x-\ell), \delta^{\prime}(x-\ell)$ are the delta-function and
its derivative at the point $\ell$, respectively. Furthermore,
\begin{equation*}\label{147}
(y,g_{\mu, h})=\frac{(\IM h)^{\frac{1}{2}}}{|\mu - h|}\ [\mu y(\ell)
-y^{\prime}(\ell)],
\end{equation*}
where $y\in \calH_+$, $g_{\mu, h}\in \calH_-$, $\calH_+ \subset L_2[\ell,+\infty) \subset \calH_-$ and the triplet of Hilbert spaces discussed above.

It was also shown in \cite{ABT} that the quasi-kernel $\hat A_\xi$ of $\RE\bA_{\mu, h}$ is given by
\begin{equation}\label{e-31}
  \left\{ \begin{array}{l}
 \hat A_\xi y=-y^{\prime\prime}+q(x)y \\
 y^{\prime}(\ell)=\xi y(\ell) \\
 \end{array} \right.,\quad \textrm{where} \quad  \xi=\frac{\mu\RE h-|h|^2}{\mu-\RE h}.
\end{equation}
Let $E=\dC$, $K_{\mu, h}{c}=cg_{\mu, h} \;(c\in \dC)$. It is clear that
\begin{equation}\label{148}
K^*_{\mu, h} y=(y,g_{\mu, h}),\quad  y\in \calH_+,
\end{equation}
and $\IM\bA_{\mu, h}=K_{\mu, h}K^*_{\mu, h}.$ Therefore, the array
\begin{equation}\label{149}
\Theta_{\mu, h}= \begin{pmatrix} \bA_{\mu, h}&K_{\mu, h}&1\cr \calH_+ \subset
L_2[\ell,+\infty) \subset \calH_-& &\dC\cr \end{pmatrix},
\end{equation}
is an L-system  with the main operator $T_h$, ($\IM h>0$) of the form \eqref{131},  the state-space operator $\bA_{\mu, h}$ of the form \eqref{137}, and  with the channel operator $K_{\mu, h}$ of the form \eqref{148}.
It was established in \cite{ArTs0}, \cite{ABT} that the transfer and impedance functions of $\Theta_{\mu, h}$ are
\begin{equation}\label{150}
W_{\Theta_{\mu, h}}(z)= \frac{\mu -h}{\mu - \overline h}\,\,
\frac{m_\infty(z)+ \overline h}{m_\infty(z)+h},
\end{equation}
and
\begin{equation}\label{1501}
V_{\Theta_{\mu, h}}(z)=\frac{\left(m_\infty(z)+\mu\right)\IM h}{\left(\mu-\RE h\right)m_\infty(z)+\mu\RE h-|h|^2}.
\end{equation}
It was shown in \cite[Section 10.2]{ABT} that if the parameters $\mu$ and $\xi$ are related via \eqref{e-31}, then the two L-systems $\Theta_{\mu, h}$ and $\Theta_{\xi, h}$ of the form \eqref{149} have the following property
\begin{equation}\label{e-35-mu-xi}
    W_{\Theta_{\mu, h}}(z)=-W_{\Theta_{\xi, h}}(z),\; V_{\Theta_{\mu, h}}(z)=-\frac{1}{V_{\Theta_{\xi, h}}(z)},\; \textrm{where} \quad \xi=\frac{\mu\RE h-|h|^2}{\mu-\RE h}.
\end{equation}
This result can be generalized as follows.
\begin{lemma}\label{l-2}
Let $\Theta_{\mu, h}$ and $\Theta_{\mu(\alpha), h}$ be two L-systems of the form \eqref{149} such that
\begin{equation}\label{e-22-alpha}
    V_{\Theta_{\mu(\alpha), h}}(z)=\frac{\cos\alpha+(\sin\alpha) V_{\Theta_{\mu, h}}(z)}{\sin\alpha-(\cos\alpha) V_{\Theta_{\mu, h}}(z)}.
\end{equation}
Then
\begin{equation}\label{e-23-mu-alpha}
    \mu(\alpha)=\frac{h(\mu-\bar h)+e^{2i\alpha}(\mu-h)\bar h}{\mu-\bar h+e^{2i\alpha}(\mu-h)}.
\end{equation}
\end{lemma}
\begin{proof}
It was shown in \cite[Section 8.3]{ABT} that if the impedance functions \break $V_{\Theta_{\mu(\alpha), h}}(z)$ and $V_{\Theta_{\mu, h}}$ are connected by the Donoghue transform \eqref{e-22-alpha} (see also \eqref{e-64-alpha}), then the corresponding transfer functions are related by
\begin{equation}\label{e-24-W}
    W_{\Theta_{\mu(\alpha), h}}(z)=(-e^{2i\alpha})\cdot W_{\Theta_{\mu, h}}(z).
\end{equation}
Combining \eqref{e-24-W} with \eqref{150} above and setting $U=-e^{2i\alpha}$ temporarily we obtain
$$
\frac{\mu(\alpha) -h}{\mu(\alpha) - \overline h}\,\,\frac{m_\infty(z)+ \overline h}{m_\infty(z)+h}=U\cdot\frac{\mu -h}{\mu - \overline h}\,\, \frac{m_\infty(z)+ \overline h}{m_\infty(z)+h},
$$
or, after canceling common factors,
$$
\frac{\mu(\alpha) -h}{\mu(\alpha) - \overline h}=U\cdot\frac{\mu -h}{\mu - \overline h}.
$$
This yields
$$
\mu(\alpha)\mu-\mu(\alpha)\bar h-h\mu+|h|^2=U\mu(\alpha)\mu-U\mu(\alpha)h-U\bar h\mu+U|h|^2.
$$
Solving the above for $\mu(\alpha)$ gives
$$
\mu(\alpha)=\frac{h\mu-|h|^2-U\bar h\mu+U|h|^2}{\mu-\bar h-U\mu+Uh}.
$$
Substituting $U=-e^{2i\alpha}$ in the above and simplifying results in \eqref{e-23-mu-alpha}.
\end{proof}
As one can easily see the value of $\xi$ in \eqref{e-35-mu-xi} follows from \eqref{e-23-mu-alpha} if one sets $\alpha=0$ and then $\xi=\mu(0)$.


\section{Realizations of $-m_\infty(z)$ and $1/m_\infty(z)$.}\label{s5}

It is known \cite{Levitan}, \cite{Na68} that the original  Weyl-Titchmarsh function $m_\infty(z)$ has a property that $(-m_\infty(z))$ is a Herglotz-Nevanlinna function. Hence, the question whether $(-m_\infty(z))$ can be realized as the impedance function of a Shr\"odinger L-system is more than relevant. The following theorem contains the answer.

\begin{theorem}\label{t-6}%
Let $\dA$ be a  symmetric Schr\"odinger operator of the form \eqref{128} with deficiency indices $(1, 1)$ and locally summable potential in $\calH=L^2[\ell, \infty).$ If $m_\infty(z)$ is the  Weyl-Titchmarsh function of $\dA$, then the Herglotz-Nevanlinna function $(-m_\infty(z))$ can be realized as the impedance function of  a Shr\"odinger L-system $\Theta_{\mu, h}$ of the form \eqref{149} with
\begin{equation}\label{e-35-h-mu}
    \mu=0\quad \textrm{and}\quad h=i.
\end{equation}

Conversely, let $\Theta_{\mu, h}$ be  a Shr\"odinger L-system of the form \eqref{149} with the symmetric operator $\dA$ such that $$V_{\Theta_{\mu, h}}(z)=-m_\infty(z),$$ for all $z\in\dC_\pm$ and $\mu\in\mathbb R\cup\{\infty\}$. Then the parameters $\mu$ and $h$ defining $\Theta_{\mu, h}$ are given by \eqref{e-35-h-mu}, i.e., $\mu=0$ and $h=i$.
\end{theorem}
\begin{proof}
Let $\Theta_{\mu, h}$ be  a Shr\"odinger L-system of the form \eqref{149} with our symmetric operator $\dA$. Then its impedance function $V_{\Theta_{\mu, h}}(z)$ is determined by formula \eqref{1501} for any $\mu\in\mathbb R\cup\{\infty\}$ and any non-real $h$. If we set $\mu=0$ and $h=i$ in \eqref{1501} we obtain
$$
V_{\Theta_{0, i}}(z)=\frac{\left(m_\infty(z)+0\right)\cdot 1}{\left(0-0\right)m_\infty(z)+0-1}=-m_\infty(z), \quad z\in\dC_\pm.
$$
Thus, $\Theta_{0, i}$ realizes $(-m_\infty(z))$ and the first part of the theorem is proved.

Conversely, let $\Theta_{\mu, h}$ be  a Shr\"odinger L-system of the form \eqref{149} such that $V_{\Theta_{\mu, h}}(z)=-m_\infty(z).$ Then \eqref{1501} implies
\begin{equation}\label{e-35-prime}
\frac{\left(m_\infty(z)+\mu\right)\IM h}{\left(\mu-\RE h\right)m_\infty(z)+\mu\RE h-|h|^2}=-m_\infty(z),
\end{equation}
for all $z\in\dC_\pm$ and $\mu\in\mathbb R\cup\{\infty\}$. In particular, if we set $\mu=0$ in the above equation we obtain
$$
\frac{\left(m_\infty(z)\right)\IM h}{\left(-\RE h\right)m_\infty(z)-|h|^2}=-m_\infty(z),\quad z\in\dC_\pm
$$
or, taking into account that $m_\infty(z)$ is not identical zero in $\dC_\pm$ (see \cite{Levitan}),
$$
\frac{\IM h}{\left(\RE h\right)m_\infty(z)+|h|^2}=1,\quad z\in\dC_\pm
$$
leading to
\begin{equation}\label{e-36-eq}
    \left(\RE h\right)m_\infty(z)+|h|^2-\IM h=0,\quad z\in\dC_\pm.
\end{equation}
Set $z=i$, then $m_\infty(i)=a-bi$, where $a$ and $b$ are real and $b>0$.  Then \eqref{e-36-eq} yields
$$
\RE h(a-bi)+|h|^2-\IM h=0,
$$
or
$$
a\RE h+|h|^2-\IM h-(b\RE h)i=0.
$$
Thus, the imaginary part of the above must be zero or $b\RE h=0$. Since $b>0$ we have $\RE h=0$ yielding
$$
(\IM h)^2-\IM h=0.
$$
Discarding the case $\IM h=0$, we obtain $\IM h=1$. Consequently, $h=i$. Substituting this value of $h$ into \eqref{e-35-prime} we get
$$
\frac{m_\infty(z)+\mu}{\mu\,m_\infty(z)-1}=-m_\infty(z),\quad z\in\dC_\pm,
$$
yielding $\mu(1+ m_\infty^2(z))=0$. This means that either $\mu=0$ or $-m_\infty(z)\equiv i$. The latter case is impossible as it can be seen from the asymptotic expansion of $m_\infty(z)$ (see \cite{DanLev90}, \cite[Chapter 2]{Levitan}) and therefore $\mu=0$. This proves the second part of the theorem.
\end{proof}
Theorem \ref{t-6} above allows us to explicitly present a Shr\"odinger L-system whose impedance function matches $-m_\infty(z)$. This is
\begin{equation}\label{e-38-sys}
    \Theta_{0, i}= \begin{pmatrix} \bA_{0, i}&K_{0, i}&1\cr \calH_+ \subset
L_2[\ell,+\infty) \subset \calH_-& &\dC\cr \end{pmatrix},
\end{equation}
where
\begin{equation}\label{e-39-star}
\begin{split}
&\bA_{0,i}\, y=-y^{\prime\prime}+q(x)y-i\,[y^{\prime}(\ell)-iy(\ell)]\,\delta^{\prime}(x-\ell), \\
&\bA^*_{0,i}\, y=-y^{\prime\prime}+q(x)y+i\,[y^{\prime}(\ell)+iy(\ell)]\,\delta^{\prime}(x-\ell),
\end{split}
\end{equation}
$K_{0, i}{c}=cg_{0, i}$, $(c\in \dC)$ and
\begin{equation}\label{e-40-g}
g_{0, i}=\delta^{\prime}(x-\ell).
\end{equation}
Clearly,
\begin{equation}\label{e-40-real}
    \RE\bA_{0,i}\,y=-y^{\prime\prime}+q(x)y-y(\ell)\delta^{\prime}(x-\ell),
\end{equation}
and (see also \eqref{e-31})
\begin{equation}\label{e-42-quasi}
    \left\{ \begin{array}{l}
 \hat A_{0,i}\, y=-y^{\prime\prime}+q(x)y \\
 y(\ell)=0 \\
 \end{array} \right.,
\end{equation}
is the quasi-kernel of $\RE\bA_{0,i}$ in \eqref{e-40-real}. Note that \eqref{e-42-quasi} defines a self-adjoint Shr\"odinger operator with \textit{Dirichlet boundary conditions}.  Also,
\begin{equation}\label{e-41-VW}
    V_{\Theta_{0, i}}(z)=-m_\infty(z)\quad \textrm{and }\quad W_{\Theta_{0, i}}(z)=\frac{1+i\,m_\infty(z)}{1-i\,m_\infty(z)}.
\end{equation}
Thus, $ V_{\Theta_{0, i}}(z)$ and hence $-m_\infty(z)$ has the following resolvent representation (see \eqref{e6-3-5})
\begin{equation}\label{e-44-resolvent}
    \begin{aligned}
-m_\infty(z)&= V_{\Theta_{0, i}}(z)=\left((\RE\bA_{0,i}-zI)^{-1}g_{0, i},g_{0,i}\right) \\&=\left(\overline{(\hat A_{0,i}-zI)^{-1}}\,\delta^{\prime}(x-\ell),\delta^{\prime}(x-\ell)\right),
    \end{aligned}
\end{equation}
where $\RE\bA_{0,i}$ is given by \eqref{e-40-real}, $g_{0, i}$ by \eqref{e-40-g}, and $\overline{(\hat A_{0,i}-zI)^{-1}}$ is the extended resolvent of the quasi-kernel $\hat A_{0,i}$ in \eqref{e-42-quasi} (see \cite{ABT}).

Now we can obtain a similar result for the function $1/m_\infty(z)$.
\begin{theorem}\label{t-7}%
Let $\dA$ be a  symmetric Schr\"odinger operator of the form \eqref{128} with deficiency indices $(1, 1)$ and locally summable potential in $\calH=L^2[\ell, \infty).$ If $m_\infty(z)$ is the  Weyl-Titchmarsh function of $\dA$, then the Herglotz-Nevanlinna function $(1/m_\infty(z))$ can be realized as the impedance function of  a Shr\"odinger L-system $\Theta_{\mu, h}$ of the form \eqref{149} with
\begin{equation}\label{e-42-h-mu}
    \mu=\infty\quad \textrm{and}\quad h=i.
\end{equation}

Conversely, let $\Theta_{\mu, h}$ be  a Shr\"odinger L-system of the form \eqref{149} with the symmetric operator $\dA$ such that $$V_{\Theta_{\mu, h}}(z)=\frac{1}{m_\infty(z)},$$ for all $z\in\dC_\pm$ and $\mu\in\mathbb R\cup\{\infty\}$. Then the parameters $\mu$ and $h$ defining $\Theta_{\mu, h}$ are given by \eqref{e-35-h-mu}, i.e., $\mu=\infty$ and $h=i$.
\end{theorem}
\begin{proof}
Let $\Theta_{\mu, h}$ be  a Shr\"odinger L-system of the form \eqref{149} with our symmetric operator $\dA$. Once again,  its impedance function $V_{\Theta_{\mu, h}}(z)$ is determined by formula \eqref{1501} for any $\mu\in\mathbb R\cup\{\infty\}$ and any non-real $h$. If we set $h=i$ and then $\mu=\infty$  in \eqref{1501} we obtain
$$
V_{\Theta_{\infty, i}}(z)=\lim_{\mu\to\infty}\frac{m_\infty(z)+\mu}{\mu\,m_\infty(z)-1}=\frac{1}{m_\infty(z)}, \quad z\in\dC_\pm.
$$
Thus, $\Theta_{\infty, i}$ realizes $(1/m_\infty(z))$ and the first part of the theorem is proved.

Conversely, let $\Theta_{\mu, h}$ be  a Shr\"odinger L-system of the form \eqref{149} such that $V_{\Theta_{\mu, h}}(z)=1/m_\infty(z).$ Then reciprocating \eqref{1501} gives
\begin{equation}\label{e-43-prime}
\frac{\left(\mu-\RE h\right)m_\infty(z)+\mu\RE h-|h|^2}{\left(m_\infty(z)+\mu\right)\IM h}=m_\infty(z),
\end{equation}
for all $z\in\dC_\pm$ and $\mu\in\mathbb R\cup\{\infty\}$. Passing to the limit in \eqref{e-43-prime} when $\mu\to\infty$ yields
$$
\frac{m_\infty(z)+\RE h}{\IM h}=m_\infty(z),\quad z\in\dC_\pm,
$$
or
\begin{equation}\label{e-44-lim}
    m_\infty(z)+\RE h=(\IM h)m_\infty(z),\quad z\in\dC_\pm.
\end{equation}
As in the proof of Theorem \ref{t-6} we set $z=i$ and $m_\infty(i)=a-bi$, where $a$ and $b$ are real and $b>0$.  Then \eqref{e-44-lim} yields
$$
(a-bi)+\RE h=\IM h(a-bi).
$$
Equating real and imaginary parts on both sides gives
$$
a+\RE h=a\IM h\quad \textrm{and} \quad b=b\IM h, \; b>0.
$$
Consequently, $\IM h=1$ and hence $a+\RE h=a$ implying $\RE h=0$. Thus $h=i$. Substituting this value of $h$ in \eqref{e-43-prime} gives
\begin{equation}\label{e-45-contr}
\frac{\mu\,m_\infty(z)-1}{m_\infty(z)+\mu}=m_\infty(z),\quad z\in\dC_\pm.
\end{equation}
If we assume that $\mu$ takes any finite real value, then \eqref{e-45-contr} leads to $1+ m_\infty^2(z)=0$ for all $z\in\dC_\pm$ or $m_\infty(z)\equiv-i$ in the upper half-plane. That is impossible (see \cite{DanLev90}, \cite{Levitan}) and we are reaching a contradiction with the assumption that $\mu$ is finite and real. Thus, the only option is $\mu=\infty$. This proves the second part of the theorem.
\end{proof}

Similarly to Theorem \ref{t-6}, Theorem \ref{t-7} above allows us to explicitly present a Shr\"odinger L-system whose impedance function matches $1/m_\infty(z)$. This is
\begin{equation}\label{e-49-sys}
    \Theta_{\infty, i}= \begin{pmatrix} \bA_{\infty, i}&K_{\infty, i}&1\cr \calH_+ \subset
L_2[\ell,+\infty) \subset \calH_-& &\dC\cr \end{pmatrix},
\end{equation}
where
\begin{equation}\label{e-50-star}
\begin{split}
&\bA_{\infty,i}\, y=-y^{\prime\prime}+q(x)y-\,[y^{\prime}(\ell)-iy(\ell)]\,\delta(x-\ell), \\
&\bA^*_{\infty,i}\, y=-y^{\prime\prime}+q(x)y-\,[y^{\prime}(\ell)+iy(\ell)]\,\delta(x-\ell),
\end{split}
\end{equation}
$K_{\infty, i}{c}=cg_{\infty, i}$, $(c\in \dC)$ and
\begin{equation}\label{e-51-g}
g_{\infty, i}=\delta(x-\ell).
\end{equation}
Clearly,
\begin{equation}\label{e-52-real}
    \RE\bA_{\infty,i}\,y=-y^{\prime\prime}+q(x)y-y'(\ell)\delta(x-\ell),
\end{equation}
and (see also \eqref{e-31})
\begin{equation}\label{e-53-quasi}
    \left\{ \begin{array}{l}
 \hat A_{\infty,i}\, y=-y^{\prime\prime}+q(x)y \\
 y'(\ell)=0 \\
 \end{array} \right.,
\end{equation}
is the quasi-kernel of $\RE\bA_{\infty,i}$ in \eqref{e-52-real}. Note that \eqref{e-53-quasi} defines a self-adjoint Shr\"odinger operator with \textit{Neumann boundary conditions}.

Also,
\begin{equation}\label{e-54-VW}
    V_{\Theta_{\infty, i}}(z)=\frac{1}{m_\infty(z)}\quad \textrm{and }\quad W_{\Theta_{\infty, i}}(z)=-\frac{1+i\,m_\infty(z)}{1-i\,m_\infty(z)}.
\end{equation}
Thus, $ V_{\Theta_{\infty, i}}(z)$ and hence $1/m_\infty(z)$ has the following resolvent representation (see \eqref{e6-3-5})
\begin{equation}\label{e-55-resolvent}
    \begin{aligned}
\frac{1}{m_\infty(z)}&= V_{\Theta_{\infty, i}}(z)=\left((\RE\bA_{\infty,i}-zI)^{-1}g_{\infty, i},g_{\infty, i}\right)\\
&=\left(\overline{(\hat A_{\infty,i}-zI)^{-1}}\,\delta(x-\ell),\delta(x-\ell)\right),
    \end{aligned}
\end{equation}
where $\RE\bA_{\infty,i}$ is given by \eqref{e-52-real}, $g_{\infty, i}$ by \eqref{e-51-g}, and $\overline{(\hat A_{\infty,i}-zI)^{-1}}$ is the extended resolvent of the quasi-kernel $\hat A_{\infty,i}$ in \eqref{e-53-quasi} (see \cite{ABT}).

We note that both L-systems $\Theta_{0,i}$ in \eqref{e-38-sys} and $\Theta_{\infty,i}$ in \eqref{e-49-sys} share the same main operator
\begin{equation}\label{e-56-T}
    \left\{ \begin{array}{l}
 T_{i}\, y=-y^{\prime\prime}+q(x)y \\
 y'(\ell)=i\,y(\ell) \\
 \end{array} \right..
\end{equation}

\section{Functions $m_\alpha(z)$}\label{s6}

Let  $\dA$ be a symmetric operator  of the form \eqref{128} with deficiency indices (1,1), generated by the differential operation $l(y)=-y^{\prime\prime}+q(x)y$. Let also $\varphi_\alpha(x,{z})$ and $\theta_\alpha(x,{z})$ be the solutions of the following Cauchy problems:
$$\left\{ \begin{array}{l}
 l(\varphi_\alpha)={z} \varphi_\alpha \\
 \varphi_\alpha(\ell,{z})=\sin\alpha \\
 \varphi'_\alpha(\ell,{z})=-\cos\alpha \\
 \end{array} \right., \qquad
\left\{ \begin{array}{l}
 l(\theta_\alpha)={z} \theta_\alpha \\
 \theta_\alpha(\ell,{z})=\cos\alpha \\
 \theta'_\alpha(\ell,{z})=\sin\alpha \\
 \end{array} \right.. $$
It is  known \cite{DanLev90}, \cite{Na68}, \cite{Ti62} that there exists an analytic in $\dC_\pm$  function $m_\alpha({z})$  for which
\begin{equation}\label{e-62-psi}
\psi(x,{z})=\theta_\alpha(x,{z})+m_\alpha({z})\varphi_\alpha(x,{z})
\end{equation}
belongs to $L_2[\ell,+\infty)$. It is easy to see that if $\alpha=\pi$, then $m_\pi({z})=m_\infty({z})$. The functions $ m_\alpha({z})$ and $m_\infty(z)$ are connected (see \cite{DanLev90}, \cite{Ti62}) by
\begin{equation}\label{e-59-LFT}
    m_\alpha({z})=\frac{\sin\alpha+m_\infty({z})\cos\alpha}{\cos\alpha-m_\infty({z})\sin\alpha}.
\end{equation}
Indeed, for any fixed ${z}\in\dC_\pm$ and any $\alpha$ we have that $\psi(x,{z})$ belongs to one-dimensional  deficiency subspace $\sN_{{z}}=\Ker (\dA^*-{z} I)$. Hence,
$$
\theta_\alpha(x,{z})+m_\alpha({z})\varphi_\alpha(x,{z})=C({z})(\theta_\pi(x,{z})+m_\infty({z})\varphi_\pi(x,{z})),\quad x\in[\ell,+\infty),
$$
where $C({z})$ is independent of $x$. Setting $x=\ell$, then differentiating and plugging in $x=\ell$ again, we obtain
$$
\begin{aligned}
\cos\alpha+m_\alpha({z})\sin\alpha&=C({z}),\\
\sin\alpha-m_\alpha({z})\cos\alpha&=-C({z})m_\infty({z}).
\end{aligned}
$$
Eliminating $C({z})$ we obtain \eqref{e-59-LFT}.

We know \cite{Na68}, \cite{Ti62} that for any real $\alpha$ the function $-m_\alpha({z})$ is a Herglotz-Nevanlinna function. Also, modifying \eqref{e-59-LFT} slightly we obtain
\begin{equation}\label{e-61-Don}
    -m_\alpha(z)=\frac{\sin\alpha+m_\infty(z)\cos\alpha}{-\cos\alpha+m_\infty(z)\sin\alpha}
    =\frac{\cos\alpha+\frac{1}{m_\infty(z)}\sin\alpha}{\sin\alpha-\frac{1}{m_\infty(z)}\cos\alpha}.
\end{equation}
Now we are going to state and prove the realization theorem for Herglotz-Nevanlinna functions $-m_\alpha(z)$ that is similar to Theorem \ref{t-6}.
\begin{theorem}\label{t-8}%
Let $\dA$ be a  symmetric Schr\"odinger operator of the form \eqref{128} with deficiency indices $(1, 1)$ and locally summable potential in $\calH=L^2[\ell, \infty).$ If $m_\alpha(z)$ is the  function of $\dA$ described in \eqref{e-62-psi}, then the Herglotz-Nevanlinna function $(-m_\alpha(z))$ can be realized as the impedance function of  a Shr\"odinger L-system $\Theta_{\mu, h}$ of the form \eqref{149} with
\begin{equation}\label{e-62-h-mu}
    \mu=\tan\alpha\quad \textrm{and}\quad h=i.
\end{equation}

Conversely, let $\Theta_{\mu, h}$ be  a Shr\"odinger L-system of the form \eqref{149} with the symmetric operator $\dA$ such that $$V_{\Theta_{\mu, h}}(z)=-m_\alpha(z),$$ for all $z\in\dC_\pm$ and $\mu\in\mathbb R\cup\{\infty\}$. Then the parameters $\mu$ and $h$ defining $\Theta_{\mu, h}$ are given by \eqref{e-62-h-mu}, i.e., $\mu=\tan\alpha$ and $h=i$.
\end{theorem}
\begin{proof}
As we have shown in Section \ref{s5}, the function $1/m_\infty(z)$ can be realized as the impedance function of the Shr\"odinger L-system $\Theta_{\infty,i}$ of the form \eqref{e-49-sys}, i.e. $1/m_\infty(z)=V_{\Theta_{\infty, i}}(z)$ for all $z\in\dC_\pm$. Consequently, \eqref{e-61-Don} yields
\begin{equation}\label{e-62-Donn}
     -m_\alpha(z)=\frac{\cos\alpha+V_{\Theta_{\infty, i}}(z)\sin\alpha}{\sin\alpha-V_{\Theta_{\infty, i}}(z)\cos\alpha}, \quad z\in\dC_\pm.
\end{equation}
But the right hand side of \eqref{e-62-Donn} is exactly the Donoghue transform \eqref{e-64-alpha} of the impedance function $V_{\Theta_{\infty, i}}(z)$ and thus represents the impedance function of a Shr\"odinger L-system with the same main operator $T_i$ of the form \eqref{e-56-T} (see \cite[Section 8.3]{ABT}). That is,
$$
-m_\alpha(z)=V_{\Theta_{\mu_\alpha, i}}(z),
$$
where $\mu_\alpha$ is a real parameter we need to find to describe the Shr\"odinger L-system $\Theta_{\mu_\alpha, i}$ realizing $-m_\alpha(z)$. In order to do that we utilize \eqref{1501}  once more. Substituting the value of $h=i$ into \eqref{1501} and applying \eqref{e-61-Don} we get
\begin{equation}\label{e-64-need}
V_{\Theta_{\mu_\alpha, i}}(z)=\frac{m_\infty(z)+\mu_\alpha}{\mu_\alpha\,m_\infty(z)-1}=-m_\alpha(z)=\frac{\sin\alpha+m_\infty(z)\cos\alpha}{-\cos\alpha+m_\infty(z)\sin\alpha},\quad z\in\dC_\pm,
\end{equation}
or, after dividing the second fraction by $\cos\alpha$,
$$
\frac{m_\infty(z)+\mu_\alpha}{\mu_\alpha\,m_\infty(z)-1}=\frac{m_\infty(z)+\tan\alpha}{m_\infty(z)\tan\alpha-1},\quad z\in\dC_\pm.
$$
Solving the above for $\mu_\alpha$ leads to that either $m_\infty(z)\equiv -i$ in $\dC_+$ (which is impossible \cite{DanLev90}, \cite{Levitan}) or
\begin{equation}\label{e-63-mu-alpha}
    \mu_\alpha=\tan\alpha.
\end{equation}
Therefore, if we set $\mu=\tan\alpha$ and $h=i$ in a Shr\"odinger L-system $\Theta_{\mu, h}$ of the form \eqref{149}, then according to the above derivations we obtain $V_{\Theta_{\mu, i}}(z)=-m_\alpha(z)$.

Conversely, if $\Theta_{\mu, h}$ is  a Shr\"odinger L-system of the form \eqref{149}  with our symmetric operator $\dA$ such that $V_{\Theta_{\mu, h}}(z)=-m_\alpha(z),$ then  \eqref{e-64-need} takes place and implies \eqref{e-63-mu-alpha}, that is $\mu=\mu_\alpha=\tan\alpha$ as shown above. The proof is complete.
\end{proof}
We note that when $\alpha=\pi$ we obtain $\mu_\alpha=0$, $m_\pi(z)=m_\infty(z)$, and the realizing Shr\"odinger L-system $\Theta_{0,i}$ is thoroughly described by \eqref{e-38-sys} in Section \ref{s5}. If $\alpha=\pi/2$, then  we get $\mu_\alpha=\infty$, $-m_\alpha(z)=1/m_\infty(z)$, and the realizing Shr\"odinger L-system $\Theta_{\infty,i}$ is given by \eqref{e-49-sys} in Section \ref{s5}. Assuming that $\alpha\in(0,\pi]$ and  neither $\alpha=\pi$ nor $\alpha=\pi/2$ we give the description of a Shr\"odinger L-system $\Theta_{\mu_\alpha,i}$ realizing $-m_\alpha(z)$ as follows.
\begin{equation}\label{e-64-sys}
    \Theta_{\tan\alpha, i}= \begin{pmatrix} \bA_{\tan\alpha, i}&K_{\tan\alpha, i}&1\cr \calH_+ \subset
L_2[\ell,+\infty) \subset \calH_-& &\dC\cr \end{pmatrix},
\end{equation}
where
\begin{equation}\label{e-65-star}
\begin{split}
&\bA_{\tan\alpha,i}\, y=l(y)-\frac{1}{\tan\alpha-i}[y^{\prime}(\ell)-iy(\ell)][(\tan\alpha)\delta(x-\ell)+\delta^{\prime}(x-\ell)], \\
&\bA^*_{\tan\alpha,i}\, y=l(y)-\frac{1}{\tan\alpha+i}\,[y^{\prime}(\ell)+iy(\ell)][(\tan\alpha)\delta(x-\ell)+\delta^{\prime}(x-\ell)],
\end{split}
\end{equation}
$K_{\tan\alpha, i}\,{c}=c\,g_{\tan\alpha, i}$, $(c\in \dC)$ and
\begin{equation}\label{e-66-g}
g_{\tan\alpha, i}=(\tan\alpha)\delta(x-\ell)+\delta^{\prime}(x-\ell).
\end{equation}
Clearly,
\begin{equation}\label{e-67-real}
    \RE\bA_{\tan\alpha,i}\,=l(y)-(\cos^2\alpha)[(\tan\alpha)y'(\ell)+y(\ell)][(\tan\alpha)\delta(x-\ell)+\delta^{\prime}(x-\ell)],
\end{equation}
and (see also \eqref{e-31})
\begin{equation}\label{e-70-quasi}
    \left\{ \begin{array}{l}
 \hat A_{\tan\alpha,i}\, y=-y^{\prime\prime}+q(x)y \\
 y(\ell)=-(\tan\alpha)\,y'(\ell) \\
 \end{array} \right.,
\end{equation}
is the quasi-kernel of $\RE\bA_{\tan\alpha,i}$ in \eqref{e-67-real}.  Also,
\begin{equation}\label{e-71-VW}
\begin{aligned}
    V_{\Theta_{\tan\alpha, i}}(z)&=-m_\alpha(z)\\
     W_{\Theta_{\tan\alpha, i}}(z)&=\frac{\tan\alpha-i}{\tan\alpha+i}\cdot\frac{m_\infty(z)-i}{m_\infty(z)+i}=(-e^{2\alpha i})\,\frac{m_\infty(z)-i}{m_\infty(z)+i}.
    \end{aligned}
\end{equation}
Thus, $ V_{\Theta_{\tan\alpha, i}}(z)$ and hence $-m_\alpha(z)$ has the following resolvent representation
\begin{equation}\label{e-72-resolvent}
    \begin{aligned}
&-m_\infty(z)= V_{\Theta_{\tan\alpha, i}}(z)=\left((\RE\bA_{\tan\alpha,i}-zI)^{-1}g_{\tan\alpha, i},g_{\tan\alpha,i}\right) \\&=\left(\overline{(\hat A_{\tan\alpha,i}-zI)^{-1}}\,[(\tan\alpha)\delta(x-\ell)+\delta^{\prime}(x-\ell)],(\tan\alpha)\delta(x-\ell)+\delta^{\prime}(x-\ell)\right),
    \end{aligned}
\end{equation}
where $\RE\bA_{\tan\alpha,i}$ is given by \eqref{e-67-real}, $g_{\tan\alpha, i}$ by \eqref{e-66-g}, and $\overline{(\hat A_{\tan\alpha,i}-zI)^{-1}}$ is the extended resolvent of the quasi-kernel $\hat A_{\tan\alpha,i}$ in \eqref{e-70-quasi} (see \cite{ABT}).

Now we are going to state and prove the realization theorem for Herglotz-Nevanlinna functions $1/m_\alpha(z)$ that is similar to Theorem \ref{t-7}.
\begin{theorem}\label{t-9}%
Let $\dA$ be a  symmetric Schr\"odinger operator of the form \eqref{128} with deficiency indices $(1, 1)$ and locally summable potential in $\calH=L^2[\ell, \infty).$ If $m_\alpha(z)$ is the  function of $\dA$ described in \eqref{e-62-psi}, then the Herglotz-Nevanlinna function $1/m_\alpha(z)$ can be realized as the impedance function of  a Shr\"odinger L-system $\Theta_{\mu, h}$ of the form \eqref{149} with
\begin{equation}\label{e-74-h-mu}
    \mu=-\cot\alpha\quad \textrm{and}\quad h=i.
\end{equation}

Conversely, let $\Theta_{\mu, h}$ be  a Shr\"odinger L-system of the form \eqref{149} with the symmetric operator $\dA$ such that $$V_{\Theta_{\mu, h}}(z)=\frac{1}{m_\alpha(z)},$$ for all $z\in\dC_\pm$ and $\mu\in\mathbb R\cup\{\infty\}$. Then the parameters $\mu$ and $h$ defining $\Theta_{\mu, h}$ are given by \eqref{e-62-h-mu}, i.e., $\mu=-\cot\alpha$ and $h=i$.
\end{theorem}
\begin{proof}
On order to prove the first part of the theorem we simply observe that the functions $(-m_\alpha(z))$ and $1/m_\alpha(z)$ are connected via similar to the middle part of \eqref{e-35-mu-xi} relation. Hence, if, according to Theorem \ref{t-6}, $(-m_\alpha(z))$ is realized by an L-system \eqref{e-64-sys} or (see \eqref{e-71-VW})
$$
-m_\alpha(z)=V_{\Theta_{\tan\alpha, i}}(z),
$$
then $1/m_\alpha(z)$ is realized \cite[Section 10.2]{ABT} by an Shr\"odinger L-system $\Theta_{\xi, h}$ with the same parameter $h=i$ and the parameter $\xi$ related to $\mu=\tan\alpha$ by the right part of \eqref{e-35-mu-xi}. This gives
$$
\xi=\frac{\mu\RE h-|h|^2}{\mu-\RE h}=\frac{-1}{\tan\alpha}=-\cot\alpha.
$$
Thus, $\Theta_{-\cot\alpha, i}$ realizes $1/m_\alpha(z)$.

 Conversely, if $\Theta_{\mu, h}$ is  a Shr\"odinger L-system of the form \eqref{149}  with the symmetric operator $\dA$ such that $V_{\Theta_{\mu, h}}(z)=1/m_\alpha(z),$ then $\Theta_{\mu, h}$ shares the same main operator with $\Theta_{\tan\alpha, i}$ and hence $h=i$. Moreover, (see \eqref{1501} and \eqref{e-59-LFT})
$$
V_{\Theta_{\mu,i}}(z)=\frac{m_\infty(z)+\mu_\alpha}{\mu_\alpha\,m_\infty(z)-1}=\frac{1}{m_\alpha(z)}=\frac{\cos\alpha-m_\infty({z})\sin\alpha}{\sin\alpha+m_\infty({z})\cos\alpha},\quad z\in\dC_\pm.
$$
Solving the above for $\mu_\alpha$ leads to that either $m_\infty(z)\equiv -i$ in $\dC_+$ (which is impossible \cite{DanLev90}, \cite{Levitan}) or
$$
    \mu=\mu_\alpha=-\cot\alpha.
$$
The proof is complete.
\end{proof}
Assuming again that $\alpha\in(0,\pi)$ and  $\alpha\ne\pi/2$ we give the description of a Shr\"odinger L-system $\Theta_{\mu_\alpha,i}$ realizing $1/m_\alpha(z)$ as follows.
\begin{equation}\label{e-75-sys}
    \Theta_{(-\cot\alpha), i}= \begin{pmatrix} \bA_{(-\cot\alpha), i}&K_{(-\cot\alpha), i}&1\cr \calH_+ \subset
L_2[\ell,+\infty) \subset \calH_-& &\dC\cr \end{pmatrix},
\end{equation}
where
\begin{equation}\label{e-76-star}
\begin{split}
&\bA_{(-\cot\alpha),i}\, y=l(y)+\frac{1}{\cot\alpha+i}[y^{\prime}(\ell)-iy(\ell)][(-\cot\alpha)\delta(x-\ell)+\delta^{\prime}(x-\ell)], \\
&\bA^*_{(-\cot\alpha),i}\, y=l(y)+\frac{1}{\cot\alpha-i}\,[y^{\prime}(\ell)+iy(\ell)][(-\cot\alpha)\delta(x-\ell)+\delta^{\prime}(x-\ell)],
\end{split}
\end{equation}
$K_{(-\cot\alpha), i}\,{c}=c\,g_{(-\cot\alpha), i}$, $(c\in \dC)$ and
\begin{equation}\label{e-77-g}
g_{(-\cot\alpha), i}=(-\cot\alpha)\delta(x-\ell)+\delta^{\prime}(x-\ell).
\end{equation}
Clearly,
\begin{equation}\label{e-78-real}
    \RE\bA_{(-\cot\alpha),i}\,=l(y)+(\sin^2\alpha)[(\cot\alpha)y'(\ell)-y(\ell)][(-\cot\alpha)\delta(x-\ell)+\delta^{\prime}(x-\ell)],
\end{equation}
and (see also \eqref{e-31})
\begin{equation}\label{e-79-quasi}
    \left\{ \begin{array}{l}
 \hat A_{(-\cot\alpha),i}\, y=-y^{\prime\prime}+q(x)y \\
 y(\ell)=(\cot\alpha)\,y'(\ell) \\
 \end{array} \right.,
\end{equation}
is the quasi-kernel of $\RE\bA_{(-\cot\alpha),i}$ in \eqref{e-67-real}.  Also,
\begin{equation}\label{e-80-VW}
\begin{aligned}
    V_{\Theta_{(-\cot\alpha), i}}(z)&=\frac{1}{m_\alpha(z)}\\
     W_{\Theta_{(-\cot\alpha), i}}(z)&=\frac{\cot\alpha+i}{\cot\alpha-i}\cdot\frac{m_\infty(z)-i}{m_\infty(z)+i}=(e^{2\alpha i})\,\frac{m_\infty(z)-i}{m_\infty(z)+i}.
    \end{aligned}
\end{equation}
Thus, $ V_{\Theta_{(-\cot\alpha), i}}(z)$ and hence $1/m_\alpha(z)$ has the following resolvent representation
\begin{equation}\label{e-81-resolvent}
    \begin{aligned}
&1/m_\alpha(z)= V_{\Theta_{(-\cot\alpha), i}}(z)=\left((\RE\bA_{(-\cot\alpha),i}-zI)^{-1}g_{(-\cot\alpha), i},g_{(-\cot\alpha),i}\right) \\&=\left(\overline{(\hat A_{(-\cot\alpha),i}-zI)^{-1}}\,[\delta^{\prime}(x-\ell)-(\cot\alpha)\delta(x-\ell)],\delta^{\prime}(x-\ell)\right.
\\&\qquad\qquad\left.-(\cot\alpha)\delta(x-\ell)\right),
    \end{aligned}
\end{equation}
where $\RE\bA_{(-\cot\alpha),i}$ is given by formula \eqref{e-78-real}, $g_{(-\cot\alpha), i}$ by \eqref{e-77-g}, and \break $\overline{(\hat A_{(-\cot\alpha),i}-zI)^{-1}}$ is the extended resolvent of the quasi-kernel $\hat A_{(-\cot\alpha),i}$ in \eqref{e-79-quasi} (see \cite{ABT}).

We conclude the section with the following result.
\begin{theorem}\label{t-7n}%
Let  $\Theta_{\mu, i}$ be  a Shr\"odinger L-system of the form \eqref{149} with $h=i$. Then there exists a unique value of $\alpha\in(0,\pi]$ such that
\begin{equation}\label{e-69-V}
    V_{\Theta_{\mu, i}}(z)=-m_\alpha(z),
\end{equation}
where $m_\alpha(z)$ is defined in \eqref{e-59-LFT}.
\end{theorem}
\begin{proof}
As we did in the proof of Theorem \ref{t-8} (see \eqref{e-64-need}) we set
$$
V_{\Theta_{\mu, i}}(z)=\frac{m_\infty(z)+\mu}{\mu\,m_\infty(z)-1}=-m_\alpha(z)=\frac{\sin\alpha+m_\infty(z)\cos\alpha}{-\cos\alpha+m_\infty(z)\sin\alpha},\quad z\in\dC_\pm,
$$
which leads to  (see \eqref{e-63-mu-alpha}) $\mu=\tan\alpha$.

Therefore, any value of $\mu\in\mathbb R\cup\{\pm\infty\}$ produces a unique value of $\alpha\in(0,\pi]$ and thus $m_\alpha(z)$ defined by \eqref{e-59-LFT}. Consequently, \eqref{e-69-V} is true.
\end{proof}
Clearly, trigonometry implies that a similar to Theorem \ref{t-7n} result takes place if one replaces \eqref{e-69-V} with $V_{\Theta_{\mu, i}}(z)=1/m_\alpha(z)$.

\section{Non-negative Schr\"odinger  operator case}\label{s7}
Now let us assume that  $\dA$ is a non-negative (i.e., $(\dA f,f) \geq 0$ for all $f \in \dom(\dA)$) symmetric operator  of the form \eqref{128} with deficiency indices (1,1), generated by the differential operation $l(y)=-y^{\prime\prime}+q(x)y$. The following theorem takes place.
\begin{theorem}[\cite{AT2009}, \cite{Ts2}, \cite{T87}]\label{t-10}
Let $\dA$ be a nonnegative symmetric Schr\"odinger operator of the form \eqref{128} with deficiency indices $(1, 1)$ and locally summable potential in $\calH=L^2[\ell,\infty).$ Consider operator $T_h$ of the form \eqref{131}.  Then
 \begin{enumerate}
\item operator $\dA$ has more than one non-negative self-adjoint extension, i.e., the Friedrichs extension $A_F$ and the Kre\u{\i}n-von Neumann extension $A_K$ do not coincide, if and only if $m_{\infty}(-0)<\infty$;
 \item operator $T_h$, ($h=\bar h$) coincides with the Kre\u{\i}n-von Neumann extension $A_K$ if and  only if $h=-m_{\infty}(-0)$;
\item operator $T_h$ is accretive if and only if
\begin{equation}\label{138}
\RE h\geq-m_\infty(-0);
\end{equation}
\item operator $T_h$, ($h\ne\bar h$) is $\beta$-sectorial if and only if  $\RE h >-m_{\infty}(-0)$ holds;
\item operator $T_h$, ($h\ne\bar h$) is accretive but not $\beta$-sectorial for any $\beta\in (0, \frac{\pi}{2})$ if and only if $\RE h=-m_{\infty}(-0)$ \item If $T_h, (\IM h>0)$ is $\beta$-sectorial,
then the exact angle $\beta$ can be calculated via
\begin{equation}\label{e10-45}
\tan\beta=\frac{\IM h}{\RE h+m_{\infty}(-0)}.
\end{equation}
\end{enumerate}
\end{theorem}
For the remainder of this paper we assume that $m_{\infty}(-0)<\infty$. Then according to Theorem \ref{t-10} above (see also \cite{Ts81}, \cite{Ts80}) we have the existence of the operator $T_h$, ($\IM h>0$) that is accretive and/or sectorial.
It  was shown in \cite{ABT} that if $T_h \;(\IM h>0)$ is an accretive Schr\"odinger operator of the form \eqref{131}, then for all real $\mu$ satisfying the following inequality
\begin{equation}\label{151}
\mu \geq \frac {(\IM h)^2}{m_\infty(-0)+\RE h}+\RE h,
\end{equation}
formulas \eqref{137}
 define the set of all accretive $(*)$-extensions $\bA_{\mu,h}$ of the operator $T_h$. Moreover, an accretive $(*)$-extensions $\bA_{\mu,h}$ of a sectorial  operator $T_h$ with exact angle of sectoriality $\beta\in(0,\pi/2)$ also preserves the same exact angle of sectoriality if and only if $\mu=+\infty$ in \eqref{137} (see \cite[Theorem 3]{BT-15}). Also, $\bA_{\mu,h}$ is accretive but not  $\beta$-sectorial for any $\beta\in(0,\pi/2)$ ($*$)-extension of $T_h$
 if and only if in \eqref{137}
\begin{equation}\label{e10-134}
\mu=\frac{(\IM h)^2}{m_\infty(-0)+\RE h}+\RE h,
\end{equation}
 (see \cite[Theorem 4]{BT-15}).
An accretive operator $T_h$ has a unique accretive $(*)$-extension $\bA_{\infty,h}\, $ if and only if
$$\RE h=-m_\infty(-0).$$
In this case this unique $(*)$-extension has the form
\begin{equation}\label{153}
\begin{aligned}
&\bA_{\infty,h}\,  y=-y^{\prime\prime}+q(x)y+[hy(\ell)-y^{\prime}(\ell)]\,\delta(x-\ell), \\
&\bA^*_{\infty,h}\,  y=-y^{\prime\prime}+q(x)y+[\overline h
y(\ell)-y^{\prime}(\ell)]\,\delta(x-\ell).
\end{aligned}
\end{equation}
Now we will see how the additional requirement of non-negativity affects the realization of functions $-m_\infty(z)$ and $1/m_\infty(z)$.
\begin{theorem}\label{t-11}%
Let $\dA$ be a non-negative symmetric Schr\"odinger operator of the form \eqref{128} with deficiency indices $(1, 1)$ and locally summable potential in $\calH=L^2[\ell, \infty).$ If $m_\infty(z)$ is the  Weyl-Titchmarsh function of $\dA$, then the L-system $\Theta_{0, i}$ of the form \eqref{e-38-sys} realizing the  function $(-m_\infty(z))$ is never accretive. The L-system $\Theta_{\infty, i}$ of the form \eqref{e-49-sys} realizing the  function $1/m_\infty(z)$ is accretive if and only if $m_\infty(-0)\ge0$.
 \end{theorem}
\begin{proof}
First, let us consider the L-system $\Theta_{0, i}$ of the form \eqref{e-38-sys} realizing the  function $-m_\infty(z)$. According to Theorem \ref{t-10} we have that the main operator $T_i$ of the form \eqref{e-56-T} is accretive if and only if \eqref{138} holds, that is $0=\RE h\geq-m_\infty(-0)$ or
\begin{equation}\label{e-77-m0}
    m_\infty(-0)\ge0.
\end{equation}
Assume that \eqref{e-77-m0} holds (if it does not, the L-system $\Theta_{0, i}$ can not be accretive \textit{a priori}). Then, applying \eqref{151} we conclude that a $(*)$-extensions $\bA_{0, i}$ of the operator $T_i$ is accretive if
$$
0=\mu \geq \frac {(\IM h)^2}{m_\infty(-0)+\RE h}+\RE h=\frac{1}{m_\infty(-0)},
$$
that contradicts our assumption \eqref{e-77-m0} since $m_{\infty}(-0)<\infty$. Thus, $\bA_{0,i}$ of the form \eqref{e-39-star} is not accretive and hence the L-system $\Theta_{0, i}$ of the form \eqref{e-38-sys} realizing the  function $-m_\infty(z)$ can not be accretive.

Now let us consider the L-system $\Theta_{\infty, i}$ of the form \eqref{e-49-sys} realizing the  function $1/m_\infty(z)$. As we have shown above the main operator $T_i$ of the form \eqref{e-56-T} is accretive if and only if \eqref{e-77-m0} holds. Then, applying \eqref{151} we conclude that a $(*)$-extensions $\bA_{\infty, i}$ of the operator $T_i$ is accretive if and only if
$$
+\infty=\mu \geq \frac {(\IM h)^2}{m_\infty(-0)+\RE h}+\RE h=\frac{1}{m_\infty(-0)},
$$
that is always true due to \eqref{e-77-m0} and the assumption of our theorem. Consequently, the L-system $\Theta_{\infty, i}$ of the form \eqref{e-49-sys} realizing the  function $1/m_\infty(z)$ is accretive.
\end{proof}

Now we are going to turn to functions $m_\alpha(z)$ described by \eqref{e-62-psi}-\eqref{e-59-LFT} and associated with the non-negative operator $\dA$ above. We will see how the parameter $\alpha$ in the definition of $m_\alpha(z)$ affects the  L-system realizing $(-m_\alpha(z))$. The following theorem answers that question.
\begin{theorem}\label{t-12}%
Let $\dA$ be a non-negative symmetric Schr\"odinger operator of the form \eqref{128} with deficiency indices $(1, 1)$ and locally summable potential in $\calH=L^2[\ell, \infty)$ and such that $m_{\infty}(-0)\ge0$. If $m_\alpha(z)$ is  described by \eqref{e-62-psi}-\eqref{e-59-LFT}, then the L-system $\Theta_{\tan\alpha, i}$ of the form \eqref{e-64-sys} realizing the  function $(-m_\alpha(z))$ is accretive if and only if
\begin{equation}\label{e-78-angles}
\tan\alpha\ge \frac{1}{m_{\infty}(-0)}.
\end{equation}
 \end{theorem}
\begin{proof}
Let $\Theta_{\tan\alpha, i}$ be the Schr\"odinger L-system  of the form \eqref{e-64-sys} realizing the  function $(-m_\alpha(z))$ and $T_i$ of the form \eqref{e-56-T} be its main operator. Clearly, if $\Theta_{\tan\alpha, i}$ is accretive, then both main operator $T_i$ and state-space operator $\bA_{\tan\alpha, i}$ of the form \eqref{e-65-star} must be accretive.
Then the condition (3) in the statement of Theorem \ref{t-10} applied to the operator $T_i$ says that it is accretive if and only if
\begin{equation}\label{e-81-acr}
   0=\RE h\geq-m_\infty(-0)\quad\textrm{ or }\quad m_\infty(-0)\ge0.
\end{equation}
Hence, the necessary condition for $\Theta_{\tan\alpha, i}$ to be accretive is $m_\infty(-0)\ge0$. Applying \eqref{151} we conclude that a $\bA_{\tan\alpha, i}$  is accretive if and only if
\begin{equation}\label{e-81-sec}
\tan\alpha \geq \frac {(\IM h)^2}{m_\infty(-0)+\RE h}+\RE h=\frac{1}{m_\infty(-0)}.
\end{equation}
\end{proof}
Note that if $m_\infty(-0)=0$ in \eqref{e-78-angles}, then $\alpha=\pi/2$ and $-m_{\frac{\pi}{2}}(z)={1}/{m_\infty(z)}$. From Theorem \ref{t-11} we know that  if $m_\infty(-0)\ge0$, then ${1}/{m_\infty(z)}$ is realized by an accretive system $\Theta_{\infty, i}$ of the form \eqref{e-49-sys}. 

Now once we established a criteria for an L-system realizing $(-m_\alpha(z))$ to be accretive, we can look into more of its properties. There are two choices for an accretive L-system $\Theta_{\tan\alpha, i}$: it is either (1) accretive sectorial or (2) accretive extremal. In the case (1) we have that $\bA_{\tan\alpha, i}$ of the form \eqref{e-65-star} is $\beta_1$-sectorial with some angle of sectoriality $\beta_1$ that can only exceed the exact angle of sectoriality $\beta$ of $T_i$. In the case (2) the state-space operator $\bA_{\tan\alpha, i}$ is extremal (not sectorial for any $\beta\in(0,\pi/2)$) and is a $(*)$-extension of $T_i$ that itself can be either $\beta$-sectorial or extremal. The following theorem describes all these possibilities.

\begin{theorem}\label{t-13}
Let $\Theta_{\tan\alpha, i}$ be the accretive L-system  realizing the  function \break $(-m_\alpha(z))$  described in Theorem \ref{t-12}. The following is true:
\begin{enumerate}
  \item if $m_{\infty}(-0)=0$, then there is only one accretive L-system $\Theta_{\infty,i}$ realizing $(-m_\alpha(z))$. This L-system is extremal and its main operator $T_i$ is extremal as well.
  \item if $m_\infty(-0)>0$, then $T_i$ is $\beta$-sectorial for $\beta\in(0,\pi/2)$ and
    \begin{enumerate}
      \item if $\tan\alpha={1}/{m_\infty(-0)}$, then $\Theta_{\tan\alpha, i}$ is extremal;
      \item if $\frac{1}{m_\infty(-0)}<\tan\alpha<+\infty$, then $\Theta_{\tan\alpha, i}$ is $\beta_1$-sectorial with $\beta_1>\beta$;
      \item if $\tan\alpha=+\infty$, then $\Theta_{\infty, i}$ is $\beta$-sectorial.
    \end{enumerate}
  \end{enumerate}
 \end{theorem}
\begin{proof}
(1) As we already noted above if $m_\infty(-0)=0$, then $\alpha=\pi/2$ and $-m_{\frac{\pi}{2}}(z)={1}/{m_\infty(z)}$. Also, the condition (5) in the statement of Theorem \ref{t-10} implies that $T_i$ is extremal since $\RE h=-m_{\infty}(-0)=0$. Thus, $\Theta_{\tan\alpha, i}$ is extremal as well and
\eqref{e-81-sec} yields that $\tan\alpha=+\infty$. Therefore, $\Theta_{\infty,i}$ is the only accretive L-system  realizing $(-m_\alpha(z))$ in this case.

\begin{figure}
  \begin{center}
  \includegraphics[width=110mm]{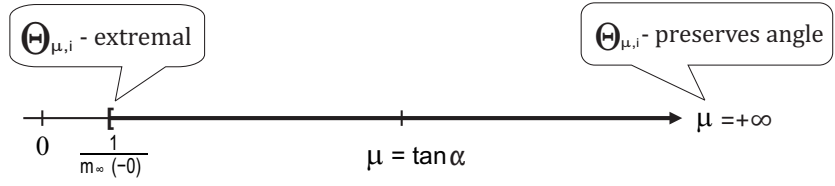}
  \caption{Accretive L-systems $\Theta_{\mu, i}$}\label{fig-1}
  \end{center}
\end{figure}

(2) If $m_\infty(-0)>0$, then the condition (4) in the statement of Theorem \ref{t-10} implies that $T_i$ is $\beta$-sectorial with $\beta\in(0,\pi/2)$ and the exact angle of sectoriality $\beta$ is given by \eqref{e10-45} as
\begin{equation}\label{e-82-sec}
\tan\beta=\frac{\IM h}{\RE h+m_{\infty}(-0)}=\frac{1}{m_\infty(-0)}\le\tan\alpha,
\end{equation}
where the last inequality is due to the fact that $\Theta_{\tan\alpha, i}$ is accretive and hence \eqref{e-81-sec} takes place. If we assume (2a), then for the L-system $\Theta_{\tan\alpha, i}$ we have
$$
\mu=\tan\alpha=\frac{1}{m_\infty(-0)}=\frac{(\IM h)^2}{m_\infty(-0)+\RE h}+\RE h,
$$
and hence according to \eqref{e10-134} we have that $\bA_{\tan\alpha, i}$ is  accretive but not  $\beta$-sectorial for any $\beta\in(0,\pi/2)$. Consequently,  $\Theta_{\tan\alpha, i}$ is extremal.

If we assume (2b), then $\mu=\tan\alpha$ is finite but strictly greater than $\tan\beta=\frac{1}{m_\infty(-0)}$ and hence \eqref{e10-134} is not true. Therefore, the accretive $\bA_{\tan\alpha, i}$ cannot be extremal and thus is $\beta_1$-sectorial for some $\beta_1\in(0,\pi/2)$. Since $\bA_{\tan\alpha, i}$ is a $(*)$-extension of the $\beta$-sectorial operator $T_i$, then $\beta_1>\beta$. Thus, $\Theta_{\tan\alpha, i}$ is $\beta_1$-sectorial with $\beta_1>\beta$.

Our last possible option is (2c) where $\mu=\tan\alpha=+\infty$. We know (see \cite[Theorem 3]{BT-15}) that in this case $\bA_{\tan\alpha, i}$ preserves the same exact angle of sectoriality as in $T_i$. As a result $\Theta_{\infty, i}$ is $\beta$-sectorial.

The proof is complete.
\end{proof}
Figure \ref{fig-1} above describes the dependence of the properties of realizing $(-m_\alpha(z))$ L-systems on the value of $\mu$ and hence $\alpha$. The bold part of the real line depicts values of $\mu=\tan\alpha$ that produce accretive L-systems $\Theta_{\mu, i}$.

The next theorem describes additional analytic properties of Herglotz-Nevanlinna functions $(-m_\infty(z))$, $1/m_\infty(z)$, and $(-m_\alpha(z))$.
\begin{theorem}\label{t-14}%
Let $\dA$ be a non-negative symmetric Schr\"odinger operator of the form \eqref{128} with deficiency indices $(1, 1)$ and locally summable potential in $\calH=L^2[\ell, \infty).$ Then:
 \begin{enumerate}
   \item the  function $1/m_\infty(z)$ is Stieltjes if and only if $m_\infty(-0)\ge0$;
   \item the  function $(-m_\infty(z))$ is never Stieltjes;\footnote{It will be shown in an upcoming paper that if $m_\infty(-0)\ge0$, then the  function $(-m_\infty(z))$ is actually inverse Stieltjes.}
   \item the  function $(-m_\alpha(z))$ given by \eqref{e-59-LFT} is Stieltjes if and only if $$0<\frac{1}{m_\infty(-0)}\le\tan\alpha.$$
 \end{enumerate}
  \end{theorem}
  \begin{proof}
  It was shown in \cite[Section 9.8]{ABT} that the impedance function of an L-system is Stieltjes if and only if this L-system is accretive. The rest of the proof immediately follows from Theorems \ref{t-11} and \ref{t-12}.
  \end{proof}
We note that the Schr\"odinger L-systems  $\Theta_{0, i}$ of the form \eqref{e-38-sys}  and  $\Theta_{\infty, i}$ of the form \eqref{e-49-sys} that we described in Theorem \ref{t-11} in this section have special properties. It was shown in \cite{ABT} (see also \cite{BT-15}) that the quasi-kernel $\hat A_0$ of $\RE\bA_{0, i}$ of the form \eqref{e-31} corresponds to the Friedrich's extension  while the quasi-kernel $\hat A_\infty$ of $\RE\bA_{\infty, i}$ corresponds to the Krein-von Neumann extension of our symmetric operator $\dA$ only in the case when $m_\infty(-0)=0$.

\section{Uniqueness of Schr\"odinger L-systems}\label{s8}

We start this section with the definition of two equal L-systems of the form \eqref{e6-3-2}.
\begin{definition}\label{d-14}
Two L-systems
\begin{equation*}
\Theta_1= \begin{pmatrix} \bA_1&K_1&\ 1\cr \calH_{1+} \subset \calH \subset
\calH_{1-}& &\dC\cr \end{pmatrix}
\end{equation*}
and
\begin{equation*}
\Theta_2= \begin{pmatrix} \bA_2&K_2&\ 1\cr \calH_{2+} \subset \calH \subset
\calH_{2-}& &\dC\cr \end{pmatrix}
\end{equation*}
are \textbf{equal} if $\calH_{1+} =\calH_{2+}$, $\bA_1=\bA_2$, and $K_1=K_2$.
\end{definition}
In this section we are going to look into uniqueness issues as applied to Schr\"odin\-ger L-systems of the form \eqref{149}. The main question to consider is when two identical impedance functions guarantee two equal (in the sense of Definition \ref{d-14}) Schr\"odinger L-systems they represent.
Suppose
\begin{equation}\label{e-75-1}
\Theta_1= \Theta_{\mu_1,h_1}= \begin{pmatrix} \bA_1&K_1&\ 1\cr \calH_{1+} \subset L_2[\ell,+\infty) \subset
\calH_{1-}& &\dC\cr \end{pmatrix}
\end{equation}
and
\begin{equation}\label{e-76-2}
\Theta_2= \Theta_{\mu_2,h_2}= \begin{pmatrix} \bA_2&K_2&\ 1\cr \calH_{2+} \subset L_2[\ell,+\infty) \subset
\calH_{2-}& &\dC\cr \end{pmatrix}
\end{equation}
be two Schr\"odin\-ger L-systems of the form \eqref{149} corresponding to two generally speaking different symmetric in $L_2[\ell,+\infty)$ operators
\begin{equation}\label{e-77-1}
 \left\{ \begin{array}{l}
 \dA_1 y=-y^{\prime\prime}+q_1(x)y \\
 y(\ell)=y^{\prime}(\ell)=0 \\
 \end{array} \right.
\end{equation}
and
\begin{equation}\label{e-77-2}
 \left\{ \begin{array}{l}
 \dA_2 y=-y^{\prime\prime}+q_2(x)y \\
 y(\ell)=y^{\prime}(\ell)=0 \\
 \end{array} \right.
\end{equation}
of the form \eqref{128}. Let also $m_{\infty,1}(z)$ and $m_{\infty,2}(z)$ be the  Weyl-Titchmarsh functions of $\dA_1$ and $\dA_2$, respectively. We will see under what conditions $V_{\Theta_1}(z)=V_{\Theta_2}(z)$ would imply that $\Theta_1=\Theta_2$.  Our first result is related to functions $m_\alpha(z)$ of the form  \eqref{e-59-LFT}.
\begin{theorem}\label{t-15}%
Let $m_{\alpha_1}(z)$ and $m_{\alpha_2}(z)$ be the functions of the form \eqref{e-59-LFT} related to the operators $\dA_1$ and $\dA_2$ of the forms \eqref{e-77-1} and \eqref{e-77-2}. Let also $\Theta_1=\Theta_{\tan\alpha_1, i}$ and $\Theta_2=\Theta_{\tan\alpha_2, i}$ be Schr\"odin\-ger L-systems of the form \eqref{e-64-sys} that realize the functions $(-m_{\alpha_1}(z))$ and $(-m_{\alpha_2}(z))$, respectively. If $m_{\infty,1}(z)=m_{\infty,2}(z)$, ($z\in\dC_\pm$), then $$V_{\Theta_1}(z)=V_{\Theta_2}(z),\quad (z\in\dC_\pm)$$ implies $\Theta_1=\Theta_2$.
\end{theorem}
  \begin{proof}
First we will show that the equality   $m_{\infty,1}(z)=m_{\infty,2}(z)$, ($z\in\dC_\pm$) in addition to
$$V_{\Theta_1}(z)=V_{\Theta_{\tan\alpha_1, i}}(z)=V_{\Theta_2}(z)=V_{\Theta_{\tan\alpha_2, i}}(z),\quad (z\in\dC_\pm)$$
yields that $\tan\alpha_1=\tan\alpha_2$. We know that according to Theorem \ref{t-8},\break $(-m_{\alpha_1}(z))=V_{\Theta_{\tan\alpha_1, i}}(z)$ and $(-m_{\alpha_2}(z))=V_{\Theta_{\tan\alpha_2, i}}(z)$ for all $z\in\dC_\pm$. Hence,
$$
 m_{\alpha_1}({z})=\frac{\sin\alpha_1+m_{\infty,1}(z)\cos\alpha_1}{\cos\alpha_1-m_{\infty,1}(z)\sin\alpha_1}=m_{\alpha_2}(z)=\frac{\sin\alpha_2+m_{\infty,2}(z)\cos\alpha_2}{\cos\alpha_2-m_{\infty,2}(z)\sin\alpha_2}.
$$
Since $m_{\infty,1}(z)=m_{\infty,2}(z)=m_\infty(z)$, then we have
$$
 \frac{\sin\alpha_1+m_{\infty}(z)\cos\alpha_1}{\cos\alpha_1-m_{\infty}(z)\sin\alpha_1}=\frac{\sin\alpha_2+m_{\infty}(z)\cos\alpha_2}{\cos\alpha_2-m_{\infty}(z)\sin\alpha_2},
$$
or, after simple calculations,
$$
\Big(1+m_{\infty}^2(z)\Big)\Big(\sin\alpha_1\cos\alpha_2-\cos\alpha_1\sin\alpha_2 \Big)=0,\quad (\forall z\in\dC_\pm).
$$
First set of parentheses can not be identical zero because it leads to $m_{\infty}(z)\equiv -i$, which is impossible as we explained earlier in the paper. Thus,
$$
\sin\alpha_1\cos\alpha_2-\cos\alpha_1\sin\alpha_2=\sin(\alpha_1-\alpha_2)=0,
$$
implying $\alpha_1=\alpha_2+\pi k$, $k\in\dZ$. Therefore, $\tan\alpha_1=\tan\alpha_2$.

Now, the fact that $m_{\infty,1}(z)=m_{\infty,2}(z)$ allows us to use the fundamental Borg-Marchenko uniqueness theorem \cite{Borg}, \cite{Mar} to conclude that the potentials $q_1(x)$ and $q_2(x)$ in \eqref{e-77-1} and \eqref{e-77-2} are the same. Taking into account that $\tan\alpha_1=\tan\alpha_2$ we have that $\Theta_1=\Theta_{\tan\alpha_1, i}$ and $\Theta_2=\Theta_{\tan\alpha_2, i}$ are equal by construction.
  \end{proof}

Now we are going to state and prove a bit more general result.
\begin{theorem}\label{t-16}%
 Let  $\Theta_1=\Theta_{\mu_1, h_1}$ and $\Theta_2=\Theta_{\mu_2, h_2}$ be Schr\"odin\-ger L-systems of the form \eqref{e-75-1} and \eqref{e-76-2} with  $\mu_1=\mu_2=\mu$ and $h_1=h_2=h$, respectively. If $V_{\Theta_1}(z)=V_{\Theta_2}(z)$, $(z\in\dC_\pm)$, then $m_{\infty,1}(z)=m_{\infty,2}(z)$, ($z\in\dC_\pm$) and $\Theta_1=\Theta_2$.
\end{theorem}
  \begin{proof}
The equality of impedance functions $V_{\Theta_1}(z)=V_{\Theta_2}(z)$ and \eqref{1501} imply
\begin{equation}\label{e-80-eqv}
\frac{m_{\infty,1}(z)+\mu}{\left(\mu-\RE h\right)m_{\infty,1}(z)+\mu\RE h-|h|^2}=\frac{m_{\infty,2}(z)+\mu}{\left(\mu-\RE h\right)m_{\infty,2}(z)+\mu\RE h-|h|^2}.
\end{equation}
Then
$$
\begin{aligned}
(m_{\infty,1}(z)&+\mu)(\left(\mu-\RE h\right)m_{\infty,2}(z)+\mu\RE h-|h|^2)\\
&=(m_{\infty,2}(z)+\mu)(\left(\mu-\RE h\right)m_{\infty,1}(z)+\mu\RE h-|h|^2),
\end{aligned}
$$
which leads to
$$
\Big(m_{\infty,1}(z)-m_{\infty,2}(z)\Big)(\mu^2-2\mu\RE\mu+|h|^2)=0,\quad (\forall z\in\dC_\pm).
$$
Assuming that $m_{\infty,1}(z)\ne m_{\infty,2}(z)$ in $\dC_\pm$ brings us to the quadratic equation in $\mu$
$$
\mu^2-2\mu\RE\mu+|h|^2=0.
$$
Applying the quadratic formula gives us $\mu=\RE h\pm(\IM h)\,i$.  This contradicts the fact that $\mu$ must be real since $\IM h\ne0$. Therefore, we arrived at a contradiction and the only logical choice is
$m_{\infty,1}(z)=m_{\infty,2}(z)$, ($z\in\dC_\pm$). Now we apply the Borg-Marchenko uniqueness theorem \cite{Borg}, \cite{Mar} to conclude that the potentials $q_1(x)$ and $q_2(x)$ in \eqref{e-77-1} and \eqref{e-77-2} are the same. Taking into account that $\mu_1=\mu_2$ and $h_1=h_2$ we have that $\Theta_1=\Theta_{\mu_1, h_1}$ and $\Theta_2=\Theta_{\mu_2, h_2}$ are equal by construction.
 \end{proof}
Summarizing the above derivations we arrive at the following uniqueness criterion for Schr\"odin\-ger L-systems with equal boundary parameters $\mu$ and $h$.
\begin{theorem}\label{t-17}%
 Let  $\Theta_1=\Theta_{\mu_1, h_1}$ and $\Theta_2=\Theta_{\mu_2, h_2}$ be Schr\"odin\-ger L-systems of the form \eqref{e-75-1} and \eqref{e-76-2} with  $\mu_1=\mu_2=\mu$ and $h_1=h_2=h$, respectively. Then  $\Theta_1=\Theta_2$ if and only if  $V_{\Theta_1}(z)=V_{\Theta_2}(z)$, $(z\in\dC_\pm)$.
\end{theorem}
  \begin{proof}
In one direction the theorem is obvious. If $\Theta_1=\Theta_2$, then clearly Definition \ref{d-14} implies that  $V_{\Theta_1}(z)=V_{\Theta_2}(z)$, $(z\in\dC_\pm)$.

In the other direction follows  from Theorem \ref{t-16}. Indeed, if $V_{\Theta_1}(z)=V_{\Theta_2}(z)$, $(z\in\dC_\pm)$, then according to Theorem \ref{t-16} we have that $m_{\infty,1}(z)=m_{\infty,2}(z)$, ($z\in\dC_\pm$) and $\Theta_1=\Theta_2$.
 \end{proof}
Now let us consider the case of Schr\"odin\-ger L-systems that share the same main operator but have different impedance functions.
 \begin{theorem}\label{t-18}%
 Let  $\Theta_1=\Theta_{\mu_1, h}$ and $\Theta_2=\Theta_{\mu_2, h}$ be Schr\"odin\-ger L-systems of the form \eqref{e-75-1} and \eqref{e-76-2} with
 \begin{equation}\label{e-80-inv}
\mu_2=\frac{\mu_1\RE h-|h|^2}{\mu_1-\RE h}.
 \end{equation}
 If
\begin{equation}\label{e-81-inv}
  V_{\Theta_1}(z)=-\frac{1}{V_{\Theta_2}(z)}, \quad (z\in\dC_\pm),
\end{equation}
  then $m_{\infty,1}(z)=m_{\infty,2}(z)$, ($z\in\dC_\pm$) and $\Theta_1$ and $\Theta_2$ share the same main operator.
\end{theorem}
  \begin{proof}
 Formula \eqref{e-81-inv} together with  \eqref{1501} yields
 $$
\frac{(m_{\infty,1}(z)+\mu_1)\IM h}{\left(\mu_1-\RE h\right)m_{\infty,1}(z)+\mu_1\RE h-|h|^2}=-\frac{\left(\mu_2-\RE h\right)m_{\infty,2}(z)+\mu_2\RE h-|h|^2}{(m_{\infty,2}(z)+\mu_2)\IM h}.
$$
Substituting \eqref{e-80-inv} into the right hand side of the above and simplifying gives us
$$
\frac{m_{\infty,1}(z)+\mu_1}{\left(\mu_1-\RE h\right)m_{\infty,1}(z)+\mu_1\RE h-|h|^2}=\frac{m_{\infty,2}(z)+\mu_1}{\left(\mu_1-\RE h\right)m_{\infty,2}(z)+\mu_1\RE h-|h|^2},
$$
that is exact analogue of \eqref{e-80-eqv} from the proof of Theorem \ref{t-16}. Following the corresponding steps of the proof of Theorem \ref{t-16} we obtain that $m_{\infty,1}(z)=m_{\infty,2}(z)$, ($z\in\dC_\pm$). Then we apply the Borg-Marchenko uniqueness theorem \cite{Borg}, \cite{Mar} and conclude that the potentials $q_1(x)$ and $q_2(x)$ in \eqref{e-77-1} and \eqref{e-77-2} are the same. Consequently, the L-systems $\Theta_1$ and $\Theta_2$ share the same main operator.
   \end{proof}
Below we provide a generalized version of Theorem \ref{t-18}.
 \begin{theorem}\label{t-19}%
 Let  $\Theta_1=\Theta_{\mu,h}$ and $\Theta_2=\Theta_{\mu(\alpha),h}$ be Schr\"odin\-ger L-systems of the form \eqref{e-75-1} and \eqref{e-76-2} with $h=h_1=h_2$, $\mu_1=\mu$, and
 \begin{equation}\label{e-86-alpha}
\mu_2=   \mu(\alpha)=\frac{h(\mu-\bar h)+e^{2i\alpha}(\mu-h)\bar h}{\mu-\bar h+e^{2i\alpha}(\mu-h)}.
 \end{equation}
 If
\begin{equation}\label{e-87-Don}
  V_{\Theta_{2}}(z)=\frac{\cos\alpha+(\sin\alpha) V_{\Theta_{1}}(z)}{\sin\alpha-(\cos\alpha) V_{\Theta_{1}}(z)},
\end{equation}
  then $m_{\infty,1}(z)=m_{\infty,2}(z)$, ($z\in\dC_\pm$) and $\Theta_1$ and $\Theta_2$ share the same main operator.

  Conversely, if two Schr\"odin\-ger L-systems $\Theta_1=\Theta_{\mu_1,h}$ and $\Theta_2=\Theta_{\mu_2,h}$ of the form \eqref{e-75-1} and \eqref{e-76-2} share the same main operator, then their impedance functions $V_{\Theta_{1}}(z)$ and $V_{\Theta_{2}}(z)$ are related by \eqref{e-87-Don} and $\mu_2$ and $\mu_1$ are connected with \eqref{e-86-alpha}.
\end{theorem}
\begin{proof}
 Our proof will be based on the method shown in the proof of Lemma \ref{l-2}.    It was shown in \cite[Section 8.3]{ABT} that if the impedance functions $V_{\Theta_{2}}(z)$ and $V_{\Theta_{1}}$ are connected by the Donoghue transform \eqref{e-87-Don} (see also \eqref{e-64-alpha} and \eqref{e-22-alpha}), then the corresponding transfer functions are related by
\begin{equation}\label{e-88-W}
    W_{\Theta_{\mu(\alpha), h}}(z)=(-e^{2i\alpha})\cdot W_{\Theta_{\mu, h}}(z).
\end{equation}
Combining \eqref{e-88-W} with \eqref{150} above and setting $U=-e^{2i\alpha}$ temporarily we obtain
\begin{equation}\label{e-89-prom}
\frac{\mu_2 -h}{\mu_2 - \overline h}\,\,\frac{m_{\infty,2}(z)+ \overline h}{m_{\infty,2}(z)+h}=U\cdot\frac{\mu -h}{\mu - \overline h}\,\, \frac{m_{\infty,1}(z)+ \overline h}{m_{\infty,1}(z)+h}.
\end{equation}
Substituting the value of $\mu_2$ from \eqref{e-86-alpha} into the first factor of left hand side above  and simplifying we obtain
\begin{equation}\label{e-90-prom}
    \frac{\mu_2 -h}{\mu_2 - \overline h}=U\cdot\frac{\mu(h-\bar h)+|h|^2-h^2}{\mu(h-\bar h)-|h|^2+\bar h^2}=U\cdot\frac{\mu-hi}{\mu-\bar hi}.
\end{equation}
Plugging \eqref{e-90-prom} into the left side of \eqref{e-89-prom} allows us to cancel $U$ and obtain
$$
\frac{\mu-hi}{\mu-\bar hi}\cdot\frac{m_{\infty,2}(z)+ \overline h}{m_{\infty,2}(z)+h}=\frac{\mu -h}{\mu - \overline h}\cdot\frac{m_{\infty,1}(z)+ \overline h}{m_{\infty,1}(z)+h},
$$
Performing further algebraic manipulations leads us to
$$
2\IM h\Big(m_{\infty,1}(z)-m_{\infty,2}(z)\Big)(\mu^2-2\mu\RE\mu+|h|^2)=0,\quad (\forall z\in\dC_\pm).
$$
Since $\IM h>0$ and, as we have shown in the proof of Theorem \ref{t-16}, the quadratic equation $\mu^2-2\mu\RE\mu+|h|^2=0$ does not have any real solutions we conclude that $m_{\infty,1}(z)=m_{\infty,2}(z)$, ($z\in\dC_\pm$). Now we apply the Borg-Marchenko uniqueness theorem \cite{Borg}, \cite{Mar} to conclude that the potentials $q_1(x)$ and $q_2(x)$ in \eqref{e-77-1} and \eqref{e-77-2} are the same. Consequently, the L-systems $\Theta_1$ and $\Theta_2$ share the same main operator.

Conversely, let two Schr\"odin\-ger L-systems $\Theta_1=\Theta_{\mu_1,h}$ and $\Theta_2=\Theta_{\mu_2,h}$ of the form \eqref{e-75-1} and \eqref{e-76-2} share the same main operator. Then according to \cite[Theorem 8.2.3]{ABT}
\begin{equation}\label{e-91-W}
    W_{\Theta_{\mu_2, h}}(z)=(-e^{2i\alpha})\cdot W_{\Theta_{\mu_1, h}}(z),
\end{equation}
for some $\alpha\in(0,\pi]$. As it was shown in \cite[Theorem 8.3.1]{ABT}, the corresponding impedance functions  $V_{\Theta_{1}}(z)$ and $V_{\Theta_{2}}(z)$ are related by  the Donoghue transform \eqref{e-87-Don} for the same value of $\alpha$. Furthermore, applying  Lemma \ref{l-2} gives us  connection \eqref{e-86-alpha} between $\mu_2$ and $\mu_1$.
\end{proof}
As we can see the result of Theorem \ref{t-18} follows from Theorem \ref{t-19} if one sets $\alpha=\pi$ in \eqref{e-86-alpha} and \eqref{e-87-Don}.

\section{Examples}

We conclude this paper with a couple of simple illustrations. Consider the differential expression with the Bessel potential
\[
l_\nu=-\frac{d^2}{dx^2}+\frac{\nu^2-1/4}{x^2},\;\; x\in [1,\infty)
\]
of order $\nu>0$ in the Hilbert space $\calH=L^2[1,\infty)$.
The minimal symmetric operator
\begin{equation}\label{ex-128}
 \left\{ \begin{array}{l}
 \dA\, y=-y^{\prime\prime}+\frac{\nu^2-1/4}{x^2}y \\
 y(1)=y^{\prime}(1)=0 \\
 \end{array} \right.
\end{equation}
 generated by this expression and boundary conditions has defect numbers $(1,1)$. 
 Consider also the operator
\begin{equation}\label{ex-79}
 \left\{ \begin{array}{l}
 T_h\, y=-y^{\prime\prime}+\frac{\nu^2-1/4}{x^2}y \\
 y'(1)=h y(1). \\
 \end{array} \right.
\end{equation}

\subsection*{Example 1} Let $\nu=1/2$. It is known \cite{ABT} that in this case the operator $\dA$ is nonnegative and
$$
m_{\infty}(z)= -{i}{\sqrt{z}}.
$$
The minimal symmetric operator in \eqref{ex-128} then becomes
$$
 \left\{ \begin{array}{l}
 \dA\, y=-y^{\prime\prime} \\
 y(1)=y^{\prime}(1)=0. \\
 \end{array} \right.
$$
The main operator $T_h$ is written for $h=i$ in \eqref{ex-79} as
\begin{equation}\label{ex-80}
 \left\{ \begin{array}{l}
 T_{i}\, y=-y^{\prime\prime} \\
 y'(1)=i\, y(1) \\
 \end{array} \right.
\end{equation}
and it will be shared by two L-systems realizing the functions $(-m_\infty(z))$  and $(1/m_\infty(z))$.
Note, that this operator $T_i$ in \eqref{ex-80} is accretive extremal (not $\beta$-sectorial for any $\beta\in(0,\pi/2)$) since $\RE h=0=m_\infty(-0)$.

 We begin by constructing an L-system $\Theta_{0, i}$ of the form \eqref{e-38-sys} that realizes
 \begin{equation}\label{e-ex-86}
-m_\infty(z)={i}{\sqrt{z}}
 \end{equation}
  according to Theorem \ref{t-6}.
 The quasi-kernel of the real part of the state-space operator of $\Theta_{0, i}$  is determined by \eqref{e-42-quasi} as follows
\begin{equation}\label{ex-81}
 \left\{ \begin{array}{l}
\hat A_{0}\, y=-y^{\prime\prime} \\
 y(1)=0.\\
 \end{array} \right.
\end{equation}
Applying \eqref{e-38-sys}--\eqref{e-40-g} to our case we obtain
\begin{equation}\label{e-ex-83-sys}
    \Theta_{0, i}= \begin{pmatrix} \bA_{0, i}&K_{0, i}&1\cr \calH_+ \subset
L_2[1,+\infty) \subset \calH_-& &\dC\cr \end{pmatrix},
\end{equation}
where
\begin{equation}\label{e-ex-84}
\begin{split}
&\bA_{0,i}\, y=-y^{\prime\prime}-i\,[y^{\prime}(1)-iy(1)]\,\delta^{\prime}(x-1), \\
&\bA^*_{0,i}\, y=-y^{\prime\prime}+i\,[y^{\prime}(1)+iy(1)]\,\delta^{\prime}(x-1),
\end{split}
\end{equation}
$K_{0, i}{c}=cg_{0, i}$, $(c\in \dC)$ and
\begin{equation}\label{e-ex-90-g}
g_{0, i}=\delta^{\prime}(x-1).
\end{equation}
As Theorem \ref{t-11} states the L-system $\Theta_{0, i}$ in \eqref{e-ex-83-sys} is not accretive. Also,
\begin{equation}\label{e-ex-91-VW}
    V_{\Theta_{0, i}}(z)={i}{\sqrt{z}}\quad \textrm{and }\quad W_{\Theta_{0, i}}(z)=\frac{1+\sqrt{z}}{1-\sqrt{z}}.
\end{equation}

Now we build an L-system $\Theta_{\infty,i}$ of the form \eqref{e-49-sys} that realizes
 \begin{equation}\label{e-ex-91}
\frac{1}{m_\infty(z)}=\frac{i}{\sqrt{z}}
 \end{equation}
  according to Theorem \ref{t-7}. The quasi-kernel of the real part of the state-space operator of $\Theta_{\infty, i}$  is given by \eqref{e-53-quasi} and in our case is
$$
 \left\{ \begin{array}{l}
\hat A_{\infty,i}\, y=-y^{\prime\prime} \\
 y'(1)=0 .\\
 \end{array} \right.
$$
Applying \eqref{e-49-sys}--\eqref{e-51-g} to our case we obtain
\begin{equation}\label{e-93-sys}
    \Theta_{\infty, i}= \begin{pmatrix} \bA_{\infty, i}&K_{\infty, i}&1\cr \calH_+ \subset
L_2[1,+\infty) \subset \calH_-& &\dC\cr \end{pmatrix},
\end{equation}
where
\begin{equation}\label{e-ex-93}
\begin{split}
&\bA_{\infty,i}\, y=-y^{\prime\prime}-\,[y^{\prime}(1)-iy(1)]\,\delta(x-1), \\
&\bA^*_{\infty,i}\, y=-y^{\prime\prime}-\,[y^{\prime}(1)+iy(1)]\,\delta(x-1),
\end{split}
\end{equation}
$K_{\infty, i}{c}=cg_{\infty, i}$, $(c\in \dC)$ and
\begin{equation}\label{e-ex-94-g}
g_{\infty, i}=\delta(x-1).
\end{equation}
By direct check using \eqref{e-52-real} one confirms that operator $\bA_{\infty,i}$ in \eqref{e-ex-93} is accretive as follows
$$
(\RE\bA_{\infty,i}\, y,y)=\|y'(x)\|^2_{L_2}\ge0.
$$
As Theorem \ref{t-11} states the L-system $\Theta_{\infty, i}$ in \eqref{e-93-sys} is extremal accretive. Also,
\begin{equation}\label{e-ex-95-VW}
    V_{\Theta_{\infty, i}}(z)=\frac{i}{\sqrt{z}}\quad \textrm{and }\quad W_{\Theta_{\infty, i}}(z)=-\frac{1+\sqrt{z}}{1-\sqrt{z}}.
\end{equation}

\subsection*{Example 2} In this example we are going to illustrate Theorem \ref{t-13} by constructing two accretive L-systems  realizing functions $(-m_\alpha(z))$ that correspond to two endpoints of the parametric interval for $\mu=\tan\alpha$ depicted in Figure \ref{fig-1}.

Let $\nu=3/2$. It is known \cite{ABT} that in this case
$$
m_{\infty}(z)= -\frac{iz-\frac{3}{2}\sqrt{z}-\frac{3}{2}i}{\sqrt{z}+i}-\frac{1}{2}=\frac{\sqrt{z}-iz+i}{\sqrt{z}+i}=1-\frac{iz}{\sqrt{z}+i}
$$
and
$$
m_{\infty}(-0)=1.
$$
The minimal symmetric operator then becomes
\begin{equation}\label{ex-e8-15}
 \left\{ \begin{array}{l}
 \dA\, y=-y^{\prime\prime}+\frac{2}{x^2}y \\
 y(1)=y^{\prime}(1)=0. \\
 \end{array} \right.
\end{equation}
The main operator $T_h$ is written for $h=i$ in \eqref{ex-79} as
\begin{equation}\label{ex-135}
 \left\{ \begin{array}{l}
 T_{i}\, y=-y^{\prime\prime}+\frac{2}{x^2}y \\
 y'(1)=i\, y(1) \\
 \end{array} \right.
\end{equation}
and it will be shared by all the family of L-systems realizing functions $(-m_\alpha(z))$ described by \eqref{e-62-psi}-\eqref{e-59-LFT}. This operator is accretive and $\beta$-sectorial since $\RE h=0>-m_\infty(-0)=-1$ with the exact angle of sectoriality given by (see \eqref{e10-45})
\begin{equation}\label{e-ex-98}
\tan\beta=\frac{\IM h}{\RE h+m_{\infty}(-0)}=\frac{1}{0+1}=1\quad\textrm{ or}\quad \beta=\frac{\pi}{4}.
\end{equation}
We will construct a family of  L-systems $\Theta_{\tan\alpha, i}$ of the form \eqref{e-64-sys} that realizes functions $(-m_\alpha(z))$ described by \eqref{e-62-psi}--\eqref{e-61-Don} as
 \begin{equation}\label{e-ex-99}
    \begin{aligned}
-m_\alpha(z)&=\frac{\cos\alpha+\frac{1}{m_\infty(z)}\sin\alpha}{\sin\alpha-\frac{1}{m_\infty(z)}\cos\alpha}=\frac{\cos\alpha+\frac{\sqrt{z}+i}{\sqrt{z}-iz+i}\sin\alpha}{\sin\alpha-\frac{\sqrt{z}+i}{\sqrt{z}-iz+i}\cos\alpha}\\
&=\frac{({\sqrt{z}-iz+i})\cos\alpha+({\sqrt{z}+i})\sin\alpha}{({\sqrt{z}-iz+i})\sin\alpha-({\sqrt{z}+i})\cos\alpha},
    \end{aligned}
 \end{equation}
  according to Theorem \ref{t-8}.
 The quasi-kernel of the real part of the state-space operator of $\Theta_{\tan\alpha, i}$  is determined by  \eqref{e-70-quasi} as
\begin{equation}\label{e-100-quasi}
    \left\{ \begin{array}{l}
 \hat A_{\tan\alpha,i}\, y=-y^{\prime\prime}+\frac{2}{x^2}y \\
 y(1)=-(\tan\alpha)\,y'(1) \\
 \end{array} \right..
\end{equation}
Using \eqref{e-64-sys} we get
\begin{equation}\label{e-101-sys}
    \Theta_{\tan\alpha, i}= \begin{pmatrix} \bA_{\tan\alpha, i}&K_{\tan\alpha, i}&1\cr \calH_+ \subset
L_2[1,+\infty) \subset \calH_-& &\dC\cr \end{pmatrix},
\end{equation}
where
\begin{equation*}\label{e-102-star}
\begin{split}
&\bA_{\tan\alpha,i}\, y=-y^{\prime\prime}+\frac{2}{x^2}y-\frac{1}{\tan\alpha-i}[y^{\prime}(1)-iy(1)][(\tan\alpha)\delta(x-1)+\delta^{\prime}(x-1)], \\
&\bA^*_{\tan\alpha,i}\, y=-y^{\prime\prime}+\frac{2}{x^2}y-\frac{1}{\tan\alpha+i}\,[y^{\prime}(1)+iy(1)][(\tan\alpha)\delta(x-1)+\delta^{\prime}(x-1)],
\end{split}
\end{equation*}
$K_{\tan\alpha, i}\,{c}=c\,g_{\tan\alpha, i}$, $(c\in \dC)$ and
\begin{equation*}\label{e-103-g}
g_{\tan\alpha, i}=(\tan\alpha)\delta(x-1)+\delta^{\prime}(x-1).
\end{equation*}
Also,
\begin{equation}\label{e-104-VW}
\begin{aligned}
    V_{\Theta_{\tan\alpha, i}}(z)&=-m_\alpha(z)=\frac{({\sqrt{z}-iz+i})\cos\alpha+({\sqrt{z}+i})\sin\alpha}{({\sqrt{z}-iz+i})\sin\alpha-({\sqrt{z}+i})\cos\alpha}\\
     W_{\Theta_{\tan\alpha, i}}(z)&=(-e^{2\alpha i})\,\frac{m_\infty(z)-i}{m_\infty(z)+i}=(-e^{2\alpha i})\,\frac{\sqrt{z}-2i\sqrt{z}+1+i}{\sqrt{z}-1+i}.
    \end{aligned}
\end{equation}
According to Theorem \ref{t-13} the L-systems  $\Theta_{\tan\alpha, i}$ in \eqref{e-101-sys} are accretive if
$$
 1=\frac{1}{m_\infty(-0)}\le\tan\alpha<+\infty.
$$
In addition, as we have shown in Theorem \ref{t-13} (see also Figure \ref{fig-1}), the realizing L-system $\Theta_{\tan\alpha, i}$ in \eqref{e-101-sys} becomes accretive extremal if
$$
\mu=\tan\alpha=\frac{1}{m_\infty(-0)}=1,
$$
and hence $\alpha=\pi/4$ makes an extremal L-system
\begin{equation}\label{e-105-sys}
    \Theta_{1, i}= \begin{pmatrix} \bA_{1, i}&K_{1, i}&1\cr \calH_+ \subset
L_2[1,+\infty) \subset \calH_-& &\dC\cr \end{pmatrix},
\end{equation}
where
\begin{equation*}\label{e-104-star}
\begin{split}
&\bA_{1,i}\, y=-y^{\prime\prime}+\frac{2}{x^2}y-\frac{1}{1-i}[y^{\prime}(1)-iy(1)][\delta(x-1)+\delta^{\prime}(x-1)], \\
&\bA^*_{1,i}\, y=-y^{\prime\prime}+\frac{2}{x^2}y-\frac{1}{1+i}\,[y^{\prime}(1)+iy(1)][\delta(x-1)+\delta^{\prime}(x-1)],
\end{split}
\end{equation*}
$K_{1, i}\,{c}=c\,g_{1, i}$, $(c\in \dC)$ and
\begin{equation*}\label{e-104-g}
g_{1, i}=\delta(x-1)+\delta^{\prime}(x-1).
\end{equation*}
It is easy to see that since in this case $\tan\alpha=m_\infty(-0)=1$, the quasi-kernel $\hat A_{1,i}$ of $\RE\bA_{1,i}$ in \eqref{e-100-quasi}  with $\alpha=\pi/4$ becomes  the Krein-von Neumann extension of  operator $\dA$ and has boundary conditions $y(1)=-y'(1)$. Also,
\begin{equation}\label{e-ex-104-VW}
\begin{aligned}
    V_{\Theta_{1, i}}(z)&=-m_{\frac{\pi}{4}}(z)=\frac{{\sqrt{z}-iz+i}+{\sqrt{z}+i}}{{\sqrt{z}-iz+i}-{\sqrt{z}+i}}=1-\frac{2}{\sqrt{z}}+2i,\\
     W_{\Theta_{1, i}}(z)&=(-e^{\frac{\pi}{2} i})\,\frac{m_\infty(z)-i}{m_\infty(z)+i}=\frac{1-i\sqrt{z}-2\sqrt{z}-i}{\sqrt{z}-1+i}.
    \end{aligned}
\end{equation}
Thus, $\Theta_{1, i}$ in \eqref{e-105-sys} represents an extremal (not $\beta$-sectorial for any $\beta\in(0,\pi/2)$) L-system with $\frac{\pi}{4}$-sectorial main operator $T_i$ of the form \eqref{ex-135}.

Now we are going to address the other end of the parametric interval described in Theorem \ref{t-13} and depicted in Figure \ref{fig-1}. According to part (2c) of Theorem \ref{t-13}, the realizing L-system $\Theta_{\tan\alpha, i}$ in \eqref{e-101-sys} preserves the angle of sectoriality and becomes $\frac{\pi}{4}$-sectorial if $\mu=\tan\alpha=+\infty$ or $\alpha=\pi/2$. Therefore the L-system
\begin{equation}\label{e-ex-105-sys}
    \Theta_{\infty, i}= \begin{pmatrix} \bA_{\infty, i}&K_{\infty, i}&1\cr \calH_+ \subset
L_2[1,+\infty) \subset \calH_-& &\dC\cr \end{pmatrix},
\end{equation}
where
\begin{equation}\label{e-ex-106}
\begin{split}
&\bA_{\infty,i}\, y=-y^{\prime\prime}+\frac{2}{x^2}y-\,[y^{\prime}(1)-iy(1)]\,\delta(x-1), \\
&\bA^*_{\infty,i}\, y=-y^{\prime\prime}+\frac{2}{x^2}y-\,[y^{\prime}(1)+iy(1)]\,\delta(x-1),
\end{split}
\end{equation}
$K_{\infty, i}{c}=cg_{\infty, i}$, $(c\in \dC)$ and
\begin{equation*}\label{e-ex-107-g}
g_{\infty, i}=\delta(x-1),
\end{equation*}
realizes the function $-m_{\frac{\pi}{2}}(z)=1/m_\infty(z)$.
Also,
\begin{equation}\label{e-ex-108-VW}
    \begin{aligned}
    V_{\Theta_{\infty, i}}(z)&=-m_{\frac{\pi}{2}}(z)=\frac{1}{m_\infty(z)}=\frac{\sqrt{z}+i}{\sqrt{z}-iz+i}\\
     W_{\Theta_{\infty, i}}(z)&=(-e^{{\pi} i})\,\frac{m_\infty(z)-i}{m_\infty(z)+i}=\frac{(1-i)\sqrt{z}-iz+1+i}{(1+i)\sqrt{z}-iz-1+i}.
     \end{aligned}
\end{equation}
This L-system $\Theta_{\infty, i}$ is clearly accretive according to Theorem \ref{t-11} which is also independently confirmed by direct evaluation
$$
(\RE\bA_{\infty,i}\, y,y)=\|y'(x)\|^2_{L^2}+2\|y(x)/x\|^2_{L^2}\ge0.
$$
The quasi-kernel $\hat A_{\infty,i}$ of $\RE\bA_{\infty,i}$ in \eqref{e-100-quasi} with $\alpha=\pi/2$ has boundary conditions $y'(1)=0$ and is not  the Krein-von Neumann extension of  $\dA$. Moreover,  according to Theorem \ref{t-13} (see also \cite[Theorem 9.8.7]{ABT}) the L-system $\Theta_{\infty, i}$ is $\frac{\pi}{4}$-sectorial. 
Taking into account that $(\IM \bA_{\infty,i}\, y,y)=|y(1)|^2,$ (see formula \eqref{146}) we obtain inequality \eqref{e8-29} with $\beta=\frac{\pi}{4}$, that is
$$(\RE\bA_{\infty,i}\, y,y)\ge |(\IM\bA_{\infty,i}\, y,y)|,$$
or
\begin{equation}\label{e8-26}
\|y'(x)\|^2_{L^2}+2\|y(x)/x\|^2_{L^2}\ge|y(1)|^2.
\end{equation}
Note that inequality \eqref{e8-26} turns into equality on $y(x)=1/x$ which confirms that the angle of sectoriality of the L-system $\Theta_{\infty, i}$ in \eqref{e-ex-105-sys} is exact and equals $\beta=\pi/4$. In addition,
we have shown that  the $\beta$-sectorial  form $(T_i y,y)$ defined on  $\dom(T_i)$  can be extended to the $\beta$-sectorial form $(\bA_{\infty,i}\, y,y)$ defined on $\calH_+=\dom(\dA^*)$ (see \eqref{ex-e8-15}--\eqref{ex-135}) having the exact (for both forms) angle of sectoriality $\beta=\pi/4$. A general problem of extending sectorial sesquilinear forms to sectorial ones was mentioned by T.~Kato in \cite{Ka}. It can also be shown (see \cite{ABT}) that function $-m_{\frac{\pi}{2}}(z)$ in \eqref{e-ex-108-VW} belongs to a sectorial class of Stieltjes functions introduced in \cite{AlTs1}.


\end{document}